\documentclass[draft,reqno]{amsart}

\usepackage[Symbol]{upgreek}

\usepackage[all,arc,curve,color,frame,line]{xypic}

\usepackage{amssymb}
\usepackage{epic}
\usepackage{mathrsfs}
\usepackage{accents}
\usepackage{listliketab}
\usepackage{stmaryrd}
\usepackage[refpage]{nomencl}

\newcommand{\bbold}{\mathbb}

\def\R { {\bbold R} }
\def\Q { {\bbold Q} }
\def\Z { {\bbold Z} }
\def\C { {\bbold C} }
\def\N { {\bbold N} }

\def \order{\operatorname{order}}

\def \exc {{\mathscr E}}
\def \ex{\operatorname{e}}

\renewcommand\epsilon{\varepsilon}

\def \d{\operatorname{d}}

\def \ev{\operatorname{e}}
\def \bar {\overline}
\def \<{\langle}
\def \>{\rangle}

\def \new{{\operatorname{new}}}

\def \dd{\operatorname{ddeg}}

\def \hat {\widehat}

\def \((  {(\!(}
\def \)) {)\!)}

\def \res{\operatorname{res}}

\def \k {{{\boldsymbol{k}}}}

\DeclareMathSymbol{\precequ}{\mathrel}{symbols}{"16}
\DeclareMathSymbol{\succequ}{\mathrel}{symbols}{"17}

\def \comp{\mathrel{-{\hskip0.06em\!\!\!\!\!\asymp}}}

\def \nasymp{\not\asymp}

\newcommand{\claim}[2][\!\!]{\medskip\noindent {\it Claim #1:} {\it #2}\medskip}
\newcommand{\case}[2][\!\!]{\medskip\noindent {\it Case #1:} {\it #2}\/}

\newtheorem{theorem}{Theorem}[section]
\newtheorem{lemma}[theorem]{Lemma}
\newtheorem{prop}[theorem]{Proposition}
\newtheorem{cor}[theorem]{Corollary}
\newtheorem*{theoremUnnumbered}{Theorem}
\newtheorem*{corUnnumbered}{Corollary}

\theoremstyle{definition}

\theoremstyle{remark}

\newtheorem*{remark}{Remark}

\newtheorem*{conjecture}{Conjecture}

\newcommand{\abs}[1]{\lvert#1\rvert}

\def \fd {{\mathfrak d}}
\def \fm {{\mathfrak m}}
\def \fn {{\mathfrak n}}
\def \fv {{\mathfrak v}}
\def \fw {{\mathfrak w}}

\def \ddeg{\operatorname{ddeg}}

\let\oldi\i
\let\oldj\j
\renewcommand\i{\relax\ifmmode{\boldsymbol{i}}\else\oldi\fi}
\renewcommand\j{\relax\ifmmode{\boldsymbol{j}}\else\oldj\fi}

\def \btau{{\boldsymbol{\tau}}}
\renewcommand\leq{\leqslant}
\renewcommand\geq{\geqslant}
\renewcommand\preceq{\preccurlyeq}
\renewcommand\succeq{\succcurlyeq}
\renewcommand\le{\leq}
\renewcommand\ge{\geq}

\DeclareMathAlphabet{\mathbf}{OML}{cmm}{b}{it}

\DeclareFontFamily{U}{fsy}{}
\DeclareFontShape{U}{fsy}{m}{n}{<->s*[.9]psyr}{}
\DeclareSymbolFont{der@m}{U}{fsy}{m}{n}
\DeclareMathSymbol{\der}{\mathord}{der@m}{182}

\DeclareSymbolFont{der@m}{U}{fsy}{m}{n}
\DeclareMathSymbol{\derdelta}{\mathord}{der@m}{100}

\newcommand\wt{\operatorname{wt}}

\newcommand\bsigma{\boldsymbol{\sigma}}

\newcommand\dwt{\operatorname{dwt}}

\newcommand\ndeg{\operatorname{ndeg}}

\DeclareSymbolFont{imag@m}{OT1}{cmr}{m}{ui}
\DeclareMathSymbol{\imag}{\mathord}{imag@m}{105}

\DeclareFontFamily{OMS}{smallo}{}
\DeclareFontShape{OMS}{smallo}{m}{n}{<->s*[.65]cmsy10}{}
\DeclareSymbolFont{smallo@m}{OMS}{smallo}{m}{n}
\DeclareMathSymbol{\smallo}{\mathord}{smallo@m}{79}

\DeclareFontFamily{OMS}{largerdot}{}
\DeclareFontShape{OMS}{largerdot}{m}{n}{<->s*[.8]cmsy10}{}
\DeclareSymbolFont{largerdot@m}{OMS}{largerdot}{m}{n}
\DeclareMathSymbol{\largerdot}{\mathord}{largerdot@m}{15}

\DeclareMathSymbol{\llambda}{\mathord}{der@m}{108}
\DeclareMathSymbol{\rrho}{\mathord}{der@m}{114}

\def \upl{\uplambda}

\begin{document}


\title{Maximal Immediate Extensions of Valued Differential Fields}

\author[Aschenbrenner]{Matthias Aschenbrenner}
\address{Department of Mathematics\\
University of California, Los Angeles\\
Los Angeles, CA 90095\\
U.S.A.}
\email{matthias@math.ucla.edu}

\author[van den Dries]{Lou van den Dries}
\address{Department of Mathematics\\
University of Illinois at Urbana-Cham\-paign\\
Urbana, IL 61801\\
U.S.A.}
\email{vddries@math.uiuc.edu}

\author[van der Hoeven]{Joris van der Hoeven}
\address{\'Ecole Polytechnique\\
91128 Palaiseau Cedex\\
France}
\email{vdhoeven@lix.polytechnique.fr}

\begin{abstract} We show that every valued differential field has
an immediate strict extension that is spherically complete. We also discuss 
the issue of uniqueness up to isomorphism of such an extension. 
\end{abstract}

\date{December 2017}

\maketitle

\section*{Introduction} 

\noindent  
In this paper a {\em valued differential field\/} is a valued field $K$ of equicharacteristic zero, equipped with a derivation 
$\der\colon K \to K$ that is continuous with respect to the valuation topology on the field. (The difference with \cite{ADH} and \cite{ADHO} is that there the definition did not include the continuity requirement.)

Let $K$ be a valued differential field. Unless specified otherwise, $\der$ is the derivation of $K$, and we let $v\colon K^\times=K\setminus\{0\} \to \Gamma=v(K^\times)$ be the valuation, with valuation ring $\mathcal{O}=\mathcal{O}_v$ and maximal ideal
$\smallo=\smallo_v$ of $\mathcal{O}$; we use the subscript 
$K$, as in $\der_K$, $v_K$, $\Gamma_K$, $\mathcal{O}_K$, $\smallo_K$, if we wish to indicate the dependence of $\der$, $v$, $\Gamma$, $\mathcal{O}$, $\smallo$ on $K$. We denote the residue field $\mathcal{O}/\smallo$ of $K$ by $\res(K)$. When the ambient $K$ is clear from the context we often write $a'$ instead of $\der(a)$
for $a\in K$, and set $a^\dagger:= a'/a$ for $a\in K^\times$. 

By \cite[Section~4.4]{ADH}, the continuity requirement on $\der$ amounts to 
the existence of a $\phi\in K^\times$ such that $\der\smallo\subseteq \phi\smallo$; the derivation of $K$ is said to be {\it small}\/ if this holds for $\phi=1$, that is, $\der\smallo\subseteq \smallo$.  By an {\it extension of $K$}\/ we
mean a valued differential field extension of $K$.
Let $L$ be an extension of $K$.
We identify $\Gamma$ in the usual way with an ordered subgroup of $\Gamma_L$ and $\res(K)$ with a subfield
of $\res(L)$, and we say that~$L$ is an {\it immediate}\/ extension of $K$ if $\Gamma=\Gamma_L$ and $\res(K)=\res(L)$.
We call the extension~$L$ of $K$  {\it strict}\/ if for every
$\phi\in K^\times$,
$$ \der\smallo\ \subseteq\ \phi\smallo\ \Rightarrow\
\der_L\smallo_L\ \subseteq\ \phi\smallo_L,  \qquad 
\der\mathcal{O}\ \subseteq\ \phi\smallo\ \Rightarrow\
\der_L\mathcal{O}_L\ \subseteq\ \phi\smallo_L.$$
With these conventions in place, our goal is to establish the following:

\begin{theoremUnnumbered}\label{imms} Every valued differential field has an immediate strict extension that is spherically complete.
\end{theoremUnnumbered}

\noindent
We consider this as a differential analogue of Krull's well-known theorem in~\cite[\S{}13]{K} that every valued field has a spherically complete immediate valued field extension. (Recall that for a valued field the geometric condition of spherical completeness is equivalent to the algebraic condition of being maximal in the sense of not having a proper immediate valued field extension.)  In our situation, strictness is analogous to the extended derivation ``preserving the norm''.  Weakening the theorem by dropping ``strict'' would still require strictness at various places in the proof, for example when using Lemma~\ref{Zw} and in coarsening arguments at the end of Section~\ref{sec:coarsening}. 

Throughout this paper $K$ is a valued differential field. For the sake of brevity we
say that $K$ has the {\it Krull property\/} if $K$ has a spherically complete immediate strict extension.
Let us first consider two trivial cases:

\medskip\noindent
{\em Case $\Gamma=\{0\}$.}\/ Then $K$ itself is a spherically complete immediate strict extension of $K$, and thus $K$ has the
Krull property.

\medskip\noindent
{\em Case $\der=0$.}\/ Take a spherically complete immediate valued field extension $L$ of the valued field $K$. Then $L$ with the trivial derivation is a spherically complete immediate
strict extension of $K$, so $K$ has the Krull property.

\medskip\noindent
Thus towards proving our main theorem we can assume $\Gamma\ne \{0\}$
and $\der\ne 0$ when convenient. We shall freely use facts
(with detailed references) from Sections 3.4, 4.1, 4.2, 4.3, 4.4, 4.5, 5.7, 6.1, 6.2, 6.3, 6.5, 6.6, 6.9, 9.1, 9.2, 10.5, and 11.1~in~\cite{ADH}.

Special cases of the main theorem 
are  in \cite{ADH}:  by \cite[Corollary~6.9.5]{ADH}, if $K$ has small derivation and $\der\mathcal O\not\subseteq\smallo$, then $K$ has a spherically complete immediate  extension with small derivation;  in \cite[Corollary~11.4.10]{ADH} we obtained spherically complete immediate extensions of
certain asymptotic fields.
What is new compared to
the proofs of these special cases? Mainly the notion of 
{\em strict extension\/}, the invariant convex subgroup $S(\der)$
of $\Gamma$, the {\em flexibility\/}
condition on $K$, and the lemmas about these (related) concepts; see 
Sections~\ref{sec:prelim}, \ref{secflex}, \ref{sec:flexlemmas}, \ref{sec:coarsening}. We also generalize in Section~\ref{evtbeh} the notion of 
{\em Newton degree\/} from \cite[11.1,~11.2]{ADH} to our setting. 
This gives us the
tools to adapt in Section~\ref{sec:immext} the proofs of 
these special cases to deriving our main theorem for $K$ such that $\Gamma^{>}$ has no least element and
$S(\der)=\{0\}$. Section~\ref{sec:coarsening} shows how that case extends to arbitrary $K$ using coarsening by $S(\der)$. 

We give special attention to asymptotic fields, a special kind of valued differential field introduced in~\cite[Section~9.1]{ADH}: $K$ is {\it asymptotic}\/
if for all nonzero $ f,g \in \smallo$,  $$f\in g\smallo\ \Longleftrightarrow\  f'\in   g'\smallo.$$  
For us, $H$-fields are asymptotic fields of particular interest, see~\cite[Sec\-tion~10.5]{ADH}: an {\it $H$-field}\/ is an {\it ordered}\/ valued differential field $K$ whose valuation ring~$\mathcal O$ is convex and such that, with $C=\{f\in K:f'=0\}$ denoting the constant field of~$K$, we have $\mathcal O=C+\smallo$, and  for all $f\in K$,
$f>C\Rightarrow{f'>0}$. Hardy fields extending $\R$ are $H$-fields.
Our theorem answers some questions about Hardy fields and $H$-fields that have been
around for some time. For example, it gives the following positive answer to
Question~2 in Matusinski~\cite{Ma}. (However, in~\cite{Ma}
the notion of $H$-field is construed too narrowly.) 
See also the remarks at the end of Sec\-tion~\ref{secflex}. 

\begin{corUnnumbered}
Each $H$-field has an immediate spherically complete $H$-field extension. 
\end{corUnnumbered}

\noindent
(Here strictness of the extension is automatic by Lemma~\ref{as1} below.)  
This corollary  follows from our main theorem in conjunction with the following: any immediate strict extension of an
asymptotic field is again asymptotic by Lemma~\ref{as2} below; and 
any immediate asymptotic extension $L$ of an $H$-field $K$ has
a unique field ordering extending that of $K$ in which $\mathcal O_L$ is convex;  equipped with this
ordering, $L$ is an $H$-field by \cite[Lemma~10.5.8]{ADH}.

\subsection*{Uniqueness} By Kaplansky~\cite{Ka}, a valued field 
$F$ of equicharacteristic zero
has up to isomorphism over $F$ a unique spherically complete
immediate valued field extension.  In Section~\ref{Uniqueness} we prove such uniqueness  in the setting of valued differential fields, but only when the valuation is discrete.
We also discuss there a conjecture from~\cite{ADH} about this, and recent progress on it. 

In Section~\ref{Nonuniqueness} we give an example of an $H$-field
where such uniqueness fails. Here we use some basic facts
related to transseries from Sections~10.4, 10.5, 13.9, and Appen\-dix~A in~\cite{ADH}. 

\subsection*{Acknowledgements}  We thank the referee for suggesting to make the paper more accessible by including 
explicit statements of some material from \cite{ADH}.

\subsection*{Notations and conventions} We borrow these notational conventions from~\cite{ADH}. For the reader's convenience
we repeat what is most needed in this paper. We set
$\N:=\{0,1,2,\dots\}$ and let $m$,~$n$ range over $\N$.

A valuation (tacitly, on a field) takes values in an ordered (additively written) abelian group $\Gamma$, where ``ordered'' here means ``totally ordered'', and for such $\Gamma$,
$$\Gamma^{<}\ :=\  \{\gamma\in \Gamma:\ \gamma<0\}, \quad \Gamma^{\le}\ :=\  \{\gamma\in \Gamma:\ \gamma\le 0\},$$
and likewise we define the subsets $\Gamma^{>}$, $\Gamma^{\ge}$, and $\Gamma^{\ne}:=\Gamma\setminus \{0\}$ of $\Gamma$.   
For $\alpha,\beta\in\Gamma$, $\alpha=o(\beta)$  means that $n\abs{\alpha}<\abs{\beta}$ for all $n\geq 1$.

For a field $E$ we set $E^\times:=E\setminus \{0\}$. Let $E$ be a valued field
with valuation $v\colon E^\times \to \Gamma_E=v(E^\times)$, valuation ring $\mathcal{O}_E$ and maximal ideal $\smallo_E$ of 
$\mathcal{O}_E$. When the ambient valued field $E$ is clear from the context, then for $a,b\in E$ we set 
\begin{align*} a\asymp b &\ :\Leftrightarrow\ va =vb, & a\preceq b&\ :\Leftrightarrow\ va\ge vb, & a\prec b &\ :\Leftrightarrow\  va>vb,\\
a\succeq b &\ :\Leftrightarrow\ b \preceq a, &
a\succ b &\ :\Leftrightarrow\ b\prec a, & a\sim b &\ :\Leftrightarrow\ a-b\prec a.
\end{align*}
It is easy to check that if $a\sim b$, then $a, b\ne 0$, and that
$\sim$ is an equivalence relation on $E^\times$; let $a^{\sim}$
be the equivalence class of an element $a\in E^\times$ with respect to~$\sim$. We use {\em pc-sequence\/} to abbreviate
{\em pseudocauchy sequence}; see \cite[Sections 2.2, 3.2]{ADH}.
Let also a valued field extension $F$ of $E$ be given. Then we identify in the usual way $\res(E)$ with a subfield of $\res(F)$, and
$\Gamma_E$ with an ordered subgroup of $\Gamma_F$.

Next, let $E$ be a differential field of characteristic $0$ (so the field $E$ is equipped with a single derivation $\der\colon E \to E$, as in \cite{ADH}). Then we have the differential ring $E\{Y\}=E[Y, Y', Y'',\dots]$ of differential polynomials in an indeterminate $Y$, and we set $E\{Y\}^{\ne}:= E\{Y\}\setminus \{0\}$. Let  
$P=P(Y)\in E\{Y\}$ have order at most $r\in \N$, that is,
$P\in E[Y,Y',\dots, Y^{(r)}]$. Then $P=\sum_{\i}P_{\i}Y^{\i}$, as in \cite[Section 4.2]{ADH}, with $\i$ ranging over tuples
$(i_0,\dots,i_r)\in \N^{1+r}$, $Y^{\i}:= Y^{i_0}(Y')^{i_1}\cdots (Y^{(r)})^{i_r}$, and the coefficients $P_{\i}$ are in $E$, and $P_{\i}\ne 0$ for only finitely many $\i$. For such $\i$ we set 
$$|\i|\ :=\ i_0+i_1+ \cdots + i_r, \qquad \|\i\|\ :=\ i_1+2i_2 + \dots + ri_r.$$
The {\em degree\/} and    the {\em weight} of $P\ne 0$ are, respectively,
$$\deg P:=\max\big\{|\i|:\ P_{\i}\ne 0\big\}\in \N, \qquad \wt P:=\max\big\{\|\i\|:\ P_{\i}\ne 0\big\}\in \N.$$
For $d\in\N$ we let $P_d:=\sum_{\abs{\i}=d} P_{\i} Y^{\i}$ be the {\it homogeneous part of degree $d$}\/ of $P$,
so $P=\sum_{d\in\N} P_d$ where $P_d=0$ for all but finitely many $d\in\N$.
We also use the decomposition $P=\sum_{\bsigma}P_{[\bsigma]}Y^{[\bsigma]}$; here
$\bsigma$ ranges over words $\bsigma=\sigma_1\cdots\sigma_d\in \{0,\dots,r\}^*$, $Y^{[\bsigma]}:=Y^{(\sigma_1)}\cdots Y^{(\sigma_d)}$, all $P_{[\bsigma]}\in E$ and $P_{[\bsigma]}\ne 0$ for only finitely many $\bsigma$, and $P_{[\bsigma]}=P_{[\pi(\bsigma)]}$ for all $\bsigma=\sigma_1\cdots \sigma_d$ and
permutations~$\pi$ of $\{1,\dots,d\}$, with
$\pi(\bsigma)=\sigma_{\pi(1)}\cdots \sigma_{\pi(d)}$. We set
$\|\bsigma\|:=\sigma_1+\cdots +\sigma_d$ for $\bsigma=\sigma_1\cdots \sigma_d$, so $\|\i\|=\|\bsigma\|$ whenever 
$Y^\i=Y^{[\bsigma]}$. We also use for $a\in E$ the 
{\em additive conjugate\/} $P_{+a}:= P(a+Y)\in E\{Y\}$ and the {\em multiplicative conjugate\/} $P_{\times a}:= P(aY)\in E\{Y\}$.  
If $P\notin E$, the {\it complexity}\/ of $P$ is the triple $(r,s,t)\in\N^3$ where $r$ is the order of~$P$, $s$ is the degree of $P$ in $Y^{(r)}$, and  $t$ is the total degree of $P$ (so $s,t\geq 1$).  
For the purpose of comparing complexities of differential polynomials we order~$\N^3$ lexicographically. Thus for $P,Q\in E\{Y\}\setminus E$, the complexity of $P$ and the complexity of $Q$ are less than the complexity of $PQ$.

For a valued differential field $K$ 
we construe the differential fraction field $K\<Y\>$ of $K\{Y\}$ as a valued differential field extension of $K$ by extending
$v\colon K^\times \to \Gamma$ to the valuation $K\<Y\>^\times \to \Gamma$ by requiring $vP=\min vP_{\i}$ for $P\in K\{Y\}^{\ne}$. 

\section{Preliminaries} \label{sec:prelim}

\noindent
We recall some basics about valued differential fields, mainly from Section~4.4 and Chapter~6 of \cite{ADH}, and add further material on compositional conjugation, strict extensions, the set
$\Gamma(\der)\subseteq \Gamma$, the convex subgroup $S(\der)$ of
$\Gamma$, and
coarsening. We finish this preliminary section 
with facts about the dominant degree of a differential polynomial as needed in the next section. 
{\em In this section $\phi$ ranges over $K^\times$}.

\subsection*{Compositional conjugation} The compositional conjugate
$K^\phi$ of $K$ is the valued differential field that has the same underlying valued field as $K$, but with derivation~$\phi^{-1}\der$. Let $L$ be an extension of $K$. Then $L^{\phi}$ extends $K^{\phi}$, and 
$$\text{$L$ strictly extends $K$}\ \Longleftrightarrow\ \text{$L^{\phi}$ strictly extends $K^{\phi}$.}$$
Therefore, $K$ has the Krull property iff $K^\phi$ has the Krull property:
$L$ is a spherically complete immediate strict extension of $K$ iff  $L^{\phi}$ is a spherically complete immediate strict extension of $K^{\phi}$. Moreover, 
$$\der\smallo\subseteq \phi\smallo\ \Longleftrightarrow\
\text{the derivation $\phi^{-1}\der$ of $K^\phi$ is small.}$$ 
Thus for the purpose of showing that $K$ has the Krull property
it suffices to deal with the case that its derivation $\der$ is small. 

\subsection*{Strict extensions} 
Suppose $\der$ is small. Then $\der\mathcal{O}\subseteq \mathcal{O}$ by \cite[Lemma~4.4.2]{ADH}, so $\der$ induces a derivation
$$a+\smallo \mapsto (a+\smallo)':=a' + \smallo$$ on the residue field $\res(K)$; the residue field of $K$ with this derivation is called the {\em differential residue field of $K$} and is denoted by $\res(K)$ as well. Note that the derivation
of $\res(K)$ is trivial iff $\der\mathcal{O}\subseteq \smallo$.

\medskip\noindent
The field $\C\(( t\)) $ of Laurent series with derivation
$\der=d/dt$ and the usual valuation, where $\mathcal{O}=\C[[t]]$ and $\smallo=t\C[[t]]$, is a valued differential field, since $\der\smallo=\mathcal{O}=t^{-1}\smallo$. It is an example of a valued differential field with $\der\mathcal{O}\subseteq \mathcal{O}$, but $\der\smallo\not\subseteq \smallo$.
On the other hand, under a mild assumption on $\Gamma$ we do
have $\der\mathcal{O}\subseteq \mathcal{O}\Rightarrow \der\smallo\subseteq \smallo$:

\begin{lemma}\label{small} Suppose $\der\smallo\subseteq \mathcal{O}$ and
$\Gamma^{>}$ has no least element. Then $\der\smallo\subseteq \smallo$.
\end{lemma}
\begin{proof} For $f\in \smallo$ we have $f=gh$ with $g,h\in \smallo$, so $f'=g'h+gh'\in \smallo$. 
\end{proof} 

\begin{lemma}\label{triv} Suppose $\der\smallo\subseteq \smallo$ and 
$\der\mathcal{O}\not\subseteq \smallo$. Then for all $\phi$: 
$\ \der\smallo\subseteq \phi\smallo\ \Leftrightarrow\ \phi\succeq 1$.
\end{lemma}
\begin{proof} From $\der\smallo\subseteq \phi\smallo$, we get
$\phi^{-1}\der\smallo\subseteq \smallo$, so the derivation $\phi^{-1}\der$ is small, and thus $\phi^{-1}\der\mathcal{O}\subseteq \mathcal{O}$, hence $\der\mathcal{O}\subseteq \phi\mathcal{O}$, which in view of $\der\mathcal{O}\not\subseteq \smallo$ gives $\phi\succeq 1$. For the converse, note that
if $\phi\succeq 1$, then $\smallo\subseteq \phi\smallo$. 
\end{proof}

\noindent
This leads easily to: 

\begin{lemma}\label{strict} Suppose $\der$ is small and the extension $L$ of $K$ has small derivation. Then the differential residue field
$\res(L)$ of $L$ is an extension of the differential residue field $\res(K)$ of $K$. If in addition $\der\mathcal{O}\not\subseteq \smallo$, then $L$ is a strict extension of $K$.
\end{lemma}

\begin{lemma}\label{strictalg} Let $L$ be an algebraic extension of $K$. Then $L$  strictly extends $K$.
\end{lemma}
\begin{proof} 
Proposition~6.2.1 of \cite{ADH} says that if the derivation of $K$ is small, then so is the derivation
of $L$. Now, if  $\der\smallo\subseteq \phi\smallo$, then  $\phi^{-1}\der$ is small, hence $\phi^{-1}\der_L$ is small,
and thus  $\der_L\smallo_L\subseteq \phi\smallo_L$. 
 Next, assume
$\der\mathcal{O}\subseteq \phi\smallo$. Then $\phi^{-1}\der$
is small and induces the trivial derivation on $\res(K)$. Hence~$\phi^{-1}\der_L$ is small, and the derivation it induces on
$\res(L)$ extends the trivial derivation on~$\res(K)$, so
is itself trivial, as $\res(L)$ is algebraic over $\res(K)$. 
Thus $\phi^{-1}\der_L\mathcal{O}_L\subseteq \smallo_L$, that is,
$\der_L\mathcal{O}_L\subseteq \phi\smallo_L$.   
\end{proof} 

\noindent
In this lemma the derivation of  $L$ is assumed to be continuous for the valuation topology, because of the meaning we assigned to {\em extension of $K$} and to
{\em valued differential field}.  In the proof of the lemma we used Proposition~6.2.1 in \cite{ADH}, but that proposition does not assume this continuity. Thus if we drop the  implicit  assumption that the derivation of $L$ is continuous, then Lemma~\ref{strictalg} goes through, with the continuity of this derivation as a consequence.

For immediate extensions, strictness reduces to a simpler
condition:

\begin{lemma}\label{strictimm} Let $L$ be an immediate extension of $K$ such that for all $\phi$, if $\der\smallo \subseteq \phi\smallo$, then
$\der_L\smallo_L \subseteq \phi\smallo_L$.
Then $L$ is a strict extension of $K$.
\end{lemma} 
\begin{proof} Suppose $\der\mathcal{O}\subseteq \phi\smallo$.
Given $f\in \mathcal{O}_L$ we have $f=g(1+\epsilon)$ with
$g\in \mathcal{O}$ and $\epsilon\in \smallo_L$, hence
$f'=g'(1+\epsilon) + g\epsilon'\in \phi\smallo_L$.
\end{proof} 

\noindent
The following related fact will also be useful:

\begin{lemma}\label{smaller} Suppose $\der$ is small and $L$ is an immediate extension of $K$ such that $\der_L\smallo_L\subseteq \mathcal{O}_L$.
Then $\der_L$ is small.
\end{lemma}
\begin{proof} If $a\in \smallo_L$, then $a=b(1+\epsilon)$ with
$b\in \smallo$, $\epsilon\in \smallo_L$, so $a'=b'(1+\epsilon) + b\epsilon'\in \smallo_L$. 
\end{proof}

\noindent
Let us record the following observations on extensions $M\supseteq L\supseteq K$:\begin{enumerate}
\item If $M\supseteq K$ is strict, then $L\supseteq K$ is strict.
\item If  $M\supseteq L$ and $L\supseteq K$ are strict, then so is $M\supseteq K$.
\item If $L$ is an elementary extension of $K$, then $L\supseteq K$ is strict.
\item Any divergent pc-sequence in $K$ pseudoconverges in some strict extension of $K$; this is an easy consequence of (3), cf.~\cite[Remark after Lemma~2.2.5]{ADH}.
\end{enumerate}

\subsection*{The set $\Gamma(\der)$}
Note that if $a, b\in K^\times$, $a\preceq b$, and 
$\smallo\subseteq a\smallo$, then $\smallo\subseteq b\smallo$.
The set $\Gamma(\der)\subseteq \Gamma$, denoted also by $\Gamma_K(\der)$ if we need to specify $K$, is defined as follows: 
$$\Gamma(\der)\ :=\ \{v\phi:\  \der\smallo\subseteq \phi\smallo\}. $$
This is a nonempty downward closed subset of $\Gamma$, with an upper bound
in $\Gamma$ if $\der\ne 0$. Moreover, $\Gamma(\der) < v(\der\smallo)$.
Lemma~\ref{triv} has a reformulation: 

\begin{cor}\label{gamder} If $\der\smallo\subseteq \smallo$ and $\der\mathcal{O}\not\subseteq \smallo$, then $\Gamma(\der)=\Gamma^{\le}$. 
\end{cor} 

\begin{lemma}\label{notmax} If 
$v\phi\in \Gamma(\der)$ is not maximal in $\Gamma(\der)$, then
$\der\mathcal{O}\subseteq \phi\smallo$.
\end{lemma}
\begin{proof} Let $a\in K^\times$ be such that
$v\phi < va\in \Gamma(\der)$. Then $a^{-1}\der$ is small, so
$a^{-1}\der\mathcal{O}\subseteq \mathcal{O}$, and thus $\der\mathcal{O}\subseteq a\mathcal{O}\subseteq \phi\smallo$.
\end{proof}

\begin{cor} Suppose $\Gamma_K(\der)$ has no largest element, $L$ extends $K$, and for all~$\phi$, if $\der\smallo \subseteq \phi\smallo$, then
$\der_L\smallo_L \subseteq \phi\smallo_L$.
Then $L$ strictly extends $K$.
\end{cor}
\begin{proof} If $v\phi\in \Gamma_K(\der)$, then $v\phi\in \Gamma_L(\der)$, but $v\phi$ is not maximal in $\Gamma_L(\der)$, and thus
$\der\mathcal{O}_L\subseteq \phi\smallo_L$ by Lemma~\ref{notmax}.
\end{proof}

\begin{lemma}\label{glgk} If $L$ strictly extends $K$ with $\Gamma_L=\Gamma$, then $\Gamma_L(\der)=\Gamma_K(\der)$.
\end{lemma}

\subsection*{The case of asymptotic fields} In this subsection we assume familiarity with Sections~6.5, 9.1, and the early parts of Section~9.2 in~\cite{ADH}. Recall   that  $K$ is said to be  asymptotic
if for all $f,g\in K$ with $0\neq f,g \prec 1$ we have $f\prec g \Longleftrightarrow f'\prec g'$.
In that case we
put $\Psi:=\big\{ v(f^\dagger):\ f\in K^\times,\ f\nasymp 1 \big\}\subseteq \Gamma$; if we need to make the dependence on~$K$ explicit we denote $\Psi$ by $\Psi_K$. We recall from \cite[Section~9.1]{ADH} that then $\Psi<v(f')$ for all $f\in \smallo$.
An asymptotic field $K$ is said to be {\it grounded}\/ if~$\Psi$ has a largest element, and
{\it ungrounded}\/ otherwise.

\begin{lemma}\label{as1} Suppose $K$ and $L$ are asymptotic fields, and $L$ is an immediate extension of $K$. Then
$L$ is a strict extension of $K$.
\end{lemma}
\begin{proof} Assume $\der\smallo\subseteq \smallo$; we show that
$\der\smallo_L\subseteq \smallo_L$. (Using Lemma~\ref{strictimm}, 
apply this to~$\phi^{-1}\der$ in the role of $\der$, for 
$v\phi\in \Gamma(\der)$.) Now 
$\der\smallo\subseteq \smallo$ means that there is no $\gamma\in \Gamma^{<}$ such that $\Psi \le \gamma$, by 
\cite[Lemma~9.2.9]{ADH}. Since $\Gamma_L=\Gamma$, we have 
$\Psi_{L}=\Psi$, and so there is no $\gamma\in \Gamma_L^{<}$ such that $\Psi_L\le \gamma$, which gives $\der\smallo_L\subseteq \smallo_L$. 
\end{proof}

\noindent
Here is a partial converse. (The proof assumes familiarity with asymptotic couples.) 

\begin{lemma}\label{as2} Suppose $K$ is asymptotic and $L$ strictly extends $K$ with $\Gamma_L=\Gamma$. If~$K$ is ungrounded or $\res(K)=\res(L)$, then $L$ is asymptotic.
\end{lemma}
\begin{proof} Let $a\in L^\times$, $a\not\asymp 1$. Then $a=bu$
with $b\in K^\times$ and $a\asymp b$, so $u\asymp 1$ and $a^\dagger=b^\dagger+ u^\dagger$. 
Using $\Gamma=\Gamma_L$ and the equivalence of (i) and (ii) in \cite[Proposition~9.1.3]{ADH} applied  to $L$, we see that
for $L$ to be asymptotic it is enough to 
show that $a^\dagger\asymp b^\dagger$, which in turn will follow from 
$u'\prec b^\dagger$ in view of $u'\asymp u^\dagger$. 

Suppose that $\Psi$ has no largest element. Take $\phi$ with $v(b^\dagger)< v(\phi)\in \Psi$.
 Then $v\phi < v(\der\smallo)$, so
$\der\smallo\subseteq \phi\smallo$, hence
$\der\smallo_L\subseteq \phi\smallo_L$, which by  \cite[Lem\-ma~4.4.2]{ADH} gives
$\der\mathcal{O}_L \subseteq \phi\mathcal{O}_L$, and thus
$u'\preceq \phi \prec b^\dagger$.

Next, suppose that $\res(K)=\res(L)$. Then in the above we could have taken $u=1+\varepsilon$ with $\epsilon\prec 1$. Then
$\der\smallo\subseteq b^\dagger\smallo$, so
$\der\smallo_L\subseteq b^\dagger\smallo_L$, hence 
$u'= \varepsilon' \prec  b^\dagger$.
\end{proof}

\noindent
For an example of a non-strict extension $L\supseteq K$ of asymptotic fields,
take $K=\R$ with the trivial valuation and trivial derivation, and
$L=\R\(( t\)) $ with the natural valuation $f\mapsto \order(f)\colon L^\times \to \Z$ given by $\order(f)=k$ for 
$$f\ =\ f_kt^k + f_{k+1}t^{k+1} + \cdots$$ with $f_k\ne 0$ and all coefficients 
$f_{k+n}\in \R$, and derivation $\der=d/dt$ given by
$\der(f)=\sum_{k\ne 0} kf_kt^{k-1}$ for $f=\sum_k f_kt^k$.
Note that $\Gamma_L(\der)=\Z^{<}$, so $0\notin \Gamma_L(\der)$. 
 
\medskip\noindent
The following isn't needed, but is somehow missing in \cite{ADH}.

\begin{lemma} If $K$ is asymptotic and $\Gamma^{>}$ has a least element, then $K$ is grounded.
\end{lemma}
\begin{proof} Suppose $f\in \smallo$, $f\ne 0$, and $v(f)=\min (\Gamma^{>})$.  Replacing $K$ by $K^\phi$ where $\phi=f^\dagger$  we arrange $v(f^\dagger)=0$. 
There is no $\gamma\in \Gamma$ with $0<\gamma< v(f)$, which in view of $v(f)=v(f')$ and $\Psi<v(\der\smallo)$ gives
$0=\max \Psi$.
\end{proof}

\noindent
Another class of valued differential fields considered more closely in 
\cite{ADH} is the class of {\em monotone\/} fields: by definition, $K$ is monotone iff $f^\dagger\preceq 1$ for all $f\in K^\times$. If $K$ is monotone, then so is any strict 
extension $L$ of~$K$ with $\Gamma_L=\Gamma$, 
by \cite[Corol\-lary~6.3.6]{ADH}. Note that $K$ has a monotone compositional conjugate iff for some 
$\phi\in K^\times$ we have $f^\dagger\preceq \phi$ for all
$f\in K^\times$. If $K$ has a monotone compositional conjugate, then clearly 
any strict 
extension $L$ of~$K$ with $\Gamma_L=\Gamma$ has as well. 

\subsection*{The stabilizer of $\Gamma(\der)$} By this we mean the convex subgroup
$$S_{K}(\der)\ :=\ \big\{\gamma\in \Gamma:\ \Gamma(\der)+\gamma=\Gamma(\der)\big\}$$
of $\Gamma$, also denoted by $S(\der)$ if $K$ is clear from the context.
Thus $$S(\der)^{\ge}\ =\ \big\{\gamma\in \Gamma^{\ge}:\ \Gamma(\der)+\gamma\subseteq\Gamma(\der)\big\}.$$
Note that $\Gamma(\der)$ is a union of cosets of $S(\der)$.
We have $\Gamma(a\der)=\Gamma(\der)+va$ and $S(\der)=S(a\der)$ for all $a\in K^\times$.  So $S(\der)$ is invariant under
compositional conjugation.  Normalizing $\der$ so that $0\in \Gamma(\der)$
has the effect that $S(\der)\subseteq \Gamma(\der)$.

\begin{lemma}
Suppose $K$ is asymptotic. Then $S(\der)=\{0\}$.
\end{lemma}
\begin{proof}
We have $\Psi<v(\der\smallo)$, so $\Psi\subseteq \Gamma(\der)$.  
Therefore, if $\gamma\in\Gamma^>$, say $\gamma=vg$, $g\in K^\times$, then $\beta:=v(g^\dagger)\in \Psi\subseteq\Gamma(\der)$, yet
$\beta+\gamma=v(g')\notin\Gamma(\der)$, hence $\gamma\notin S(\der)$.
\end{proof}

\noindent
The following is also easy to verify. 

\begin{lemma}\label{supremum} If $\Gamma(\der)$ has a supremum in $\Gamma$, 
then $S(\der)=\{0\}$.
\end{lemma}

\noindent
In particular, if $\Gamma(\der)$ has a maximum, then $S(\der)=\{0\}$. If $\der\ne 0$ and $\Gamma$ is archimedean, then clearly also $S(\der)=\{0\}$.
In Section~\ref{secflex} we need the following:  

\begin{lemma}\label{epsgamdel} Suppose $S(\der)=\{0\}$. Then for any
$\epsilon\in \Gamma^{>}$ there are $\gamma\in \Gamma(\der)$ and
$\delta\in \Gamma\setminus\Gamma(\der)$ such that 
$\delta-\gamma\le\epsilon$. 
\end{lemma}
\begin{proof} Let $\epsilon\in \Gamma^{>}$. Then $\epsilon\notin S(\der)$, so we get $\gamma\in \Gamma(\der)$ with 
$\delta:=\gamma+\epsilon\notin \Gamma(\der)$.  
\end{proof} 

\noindent
For cases where $S(\der)\ne \{0\}$, let $\k$ be a field of characteristic~$0$ with a valuation $w\colon \k\to \Delta=w(\k^\times)$, and let $K=\k\(( t\)) $ be the field of Laurent series over $\k$.
Then we have the valuation $f\mapsto \order(f)\colon K^\times \to \Z$,
where $\order(f)=k$ means that $f=f_kt^k + f_{k+1}t^{k+1} + \cdots$
with $f_k\ne 0$ and all coefficients $f_{k+n}\in \k$. We combine these two valuations into a single valuation $v\colon K^\times \to \Gamma$
extending the valuation $w$ on $\k$,
with $\Gamma$ having $\Delta$ and $\Z$ as ordered subgroups,
$\Delta$ convex in $\Gamma$, and $\Gamma=\Delta+ \Z$;
it is given by $v(f)=w(f_k) + k$, with $k=\order(f)$. 
Next, we equip $K$ with the derivation $\der=t\cdot d/dt$ given by
$\der(f)=\sum_k kf_kt^k$ for $f=\sum_k f_kt^k$. Then $K$ with the valuation $v$ and the derivation $\der$ is a (monotone) valued differential field, with $\der\smallo=\big\{f\in K:\ \order(f)\ge 1\big\}$. It follows easily that
$$\Gamma(\der)\ =\ 
\{\gamma\in \Gamma:\ \gamma\le \delta \text{ for some }\delta\in \Delta\},$$
and thus $S(\der)=\Delta$.

\subsection*{Coarsening} We begin with reminders
about coarsening from \cite[Sections~3.4, 4.4]{ADH}. Let $\Delta$ be a convex subgroup of $\Gamma$. This yields the ordered abelian quotient group $\dot{\Gamma}=\Gamma/\Delta$ of $\Gamma$, with the coarsened valuation 
$$\dot{v}\ =\ v_{\Delta}\ \colon\  K^\times \to \dot{\Gamma}, \qquad \dot{v}(a)\ :=\  v(a)+\Delta$$
on the underlying field of $K$. The $\Delta$-coarsening $K_{\Delta}$ of $K$ is the valued differential field with the same underlying differential field
as $K$, but with valuation $\dot{v}$. Its valuation ring is 
$$\dot{\mathcal{O}}\ =\ \{a\in K:\ va\ge \delta \text{ for some }\delta\in \Delta\}\ \supseteq\ \mathcal{O},$$ with maximal ideal
$$\dot{\smallo}\ =\ \{a\in K:\ va> \Delta\}\ \subseteq\ \smallo.$$
The residue field 
$$\dot{K}\ :=\ \res(K_{\Delta})\ =\ \dot{\mathcal{O}}/\dot{\smallo}$$ is itself a valued field with valuation $v\colon \dot{K}^\times \to \Delta$ given by 
$$v(a+\dot{\smallo})\ :=\ v(a)\qquad\text{ for $a\in \dot{\mathcal{O}}\setminus \dot{\smallo}$,}$$
and with valuation ring $\{a+\dot{\smallo}: a\in \mathcal{O}\}$.
We identify $\res(K)$ with $\res(\dot{K})$ 
via $\res(a)\mapsto \res(a+\dot{\smallo})$ for 
$a\in \mathcal{O}$. 
The following is \cite[Corollary~4.4.4]{ADH}:

\begin{lemma}\label{lem:4.4.4}
If $\der\smallo\subseteq\smallo$ then $\der\dot{\smallo}\subseteq\dot{\smallo}$.
\end{lemma}

\noindent
Suppose the derivation $\der$ of $K$ is small. Then $\der$ is also small as a derivation of $K_{\Delta}$, and the derivation on~$\dot{K}$ induced by the derivation of $K_{\Delta}$ is small as well. This derivation on~$\dot{K}$ then induces the same derivation on $\res(\dot{K})$ as $\der$ on $K$
induces on $\res(K)$. The operation of coarsening commutes with
compositional conjugation: $(K^\phi)_{\Delta}$ and $(K_{\Delta})^\phi$ are the same valued differential field, to be denoted by $K_{\Delta}^\phi$.

\medskip\noindent
The next lemma describes the downward closed subset $\dot{\Gamma}(\der)$ of $\dot{\Gamma}$
almost completely in terms of $\Gamma(\der)$
and the canonical map $\pi\colon \Gamma\to \dot{\Gamma}$. Let $\alpha$ range over $\Gamma$.

\begin{lemma}\label{coarseder1} If $\der\dot{\smallo}\subseteq \phi\dot{\smallo}$, then $v\phi\in \bigcap_{\alpha>\Delta}\Gamma(\der)+\alpha$. As a consequence we have either $\dot{\Gamma}(\der)=\pi \Gamma(\der)$, or $\dot{\Gamma}(\der)=\pi\Gamma(\der) \cup \{\dot{\mu}\}$ with $\dot{\mu}=\max \dot{\Gamma}(\der)$.
\end{lemma}
\begin{proof}  Suppose $\der\dot{\smallo}\subseteq \phi\dot{\smallo}$.
 Then $\phi^{-1}\der\dot{\smallo}\subseteq \dot{\smallo}$, so
 $\phi^{-1}\der\dot{\mathcal{O}}\subseteq \dot{\mathcal{O}}$,
 hence $\der\dot{\mathcal{O}}\subseteq \phi\dot{\mathcal{O}}$.
 For $a\in K^\times$ with $\alpha=va>\Delta$ we have
 $a\dot{\mathcal{O}}\subseteq \smallo$, so
 $\der\dot{\mathcal{O}}\subseteq  \phi\dot{\mathcal{O}}\subseteq \phi a^{-1}\smallo$, and
 thus $\der\smallo\subseteq \phi a^{-1}\smallo$, which gives
 $v\phi-\alpha\in \Gamma(\der)$. We conclude that
 $$v\phi\in \bigcap_{\alpha>\Delta} \Gamma(\der)+\alpha.$$
It follows from Lemma~\ref{lem:4.4.4}  that $\pi\Gamma(\der)\subseteq \dot{\Gamma}(\der)$. 
Suppose that $v\phi> \Gamma(\der)+\Delta$. Then by the above,
 $\dot{v}\phi-\dot{\alpha}\in \pi\Gamma(\der)$ for all $\dot{\alpha}\in \dot{\Gamma}^{>}$. If $\pi\Gamma(\der)$ has no largest element, then we get $\dot{v}\phi= \sup \pi\Gamma(\der)$.
If $\pi \Gamma(\der)$ has a largest element, then
$\dot{v}\phi-\max\pi\Gamma(\der)$ must be the least positive
element of $\dot{\Gamma}^{>}$, and $\dot{\Gamma}(\der)= \pi\Gamma(\der)\cup\{\dot{v}\phi\}$.
\end{proof} 

\noindent
The following will be needed in deriving Proposition~\ref{coarseder3}:

\begin{lemma}\label{coarseder2} Let $L$ be an immediate strict extension of $K$ such that
$\res(L_{\Delta})=\res(K_{\Delta})$. Then $L_{\Delta}$ is a strict extension of $K_{\Delta}$.
\end{lemma}
\begin{proof} Note that $L_{\Delta}$ is an immediate extension of $K_{\Delta}$. Suppose $\der\dot{\smallo}\ \subseteq\ \phi\dot{\smallo}$; applying Lemmas~\ref{strictimm} and ~\ref{smaller} to the extension $L_{\Delta}$ of $K_{\Delta}$, it suffices to derive from this assumption that $\der\dot{\smallo}_L\ \subseteq\ \phi\dot{\mathcal{O}}_L$. The proof of Lemma~\ref{coarseder1} gives $\der\smallo\subseteq \phi a^{-1}\smallo$ for every $a\in K^\times$ with $va>\Delta$, and so $\der\smallo_L\subseteq \phi a^{-1}\smallo_L$ for such $a$. Thus for $f\in \smallo_L$ we have
$v(f')> v\phi-\alpha$ for all $\alpha>\Delta$, so
$v(f'/\phi)> -\alpha$ for all $\alpha> \Delta$, that is,
$f'/\phi\in \dot{\mathcal{O}}_L$, so $f'\in \phi\dot{\mathcal{O}}_L$. This shows $\der\dot{\smallo}_L\subseteq \der \smallo_L\subseteq \phi\dot{\mathcal{O}}_L$, as desired.
\end{proof}

\subsection*{Dominant degree} We summarize here from \cite[Section~6.6]{ADH} what we need about the dominant part
and dominant degree of a differential polynomial and its behavior under additive and multiplicative conjugation. We give the  definitions, but refer to~\cite[Section~6.6]{ADH} for the proofs. {\em In this subsection we assume the derivation $\der$ of $K$ is small, and we choose for every $P\in K\{Y\}^{\ne}$ an element $\fd_P\in K^\times$ with $\fd_P\asymp P$, such that $\fd_P=\fd_Q$
whenever $P\sim Q$, $P,Q\in K\{Y\}^{\ne}$. Let 
$P\in K\{Y\}^{\ne}$}.

\medskip\noindent
We have $\fd_P^{-1}P\asymp 1$, in particular, $\fd_P^{-1}P\in \mathcal{O}\{Y\}$, and we define the {\it dominant part}\/~$D_P\in
\res(K)\{Y\}^{\ne}$ to be the image of $\fd_P^{-1}P$ under the natural differential ring morphism $\mathcal{O}\{Y\} \to \res(K)\{Y\}$.
Note that $\deg D_P\le \deg P$. The 
{\it dominant degree}\/ of $P$ is defined to be the natural number $\ddeg P:=\deg D_P$; unlike $D_P$ it does not depend on the choice of the elements $\fd_P\in K^\times$. Given also  $Q\in K\{Y\}^{\ne}$ we have 
$\ddeg PQ = \ddeg P + \ddeg Q$. 
If $f\preceq 1$ in an extension~$L$ of $K$ with small derivation satisfies 
$P(f)=0$, then $D_P(f+\smallo_L)=0$ and thus $\ddeg P \geq 1$. 

\begin{lemma} If $a\in K$ and $a\preceq 1$, then 
$\ddeg P_{+a}=\ddeg P$.
\end{lemma}

\begin{lemma} Let $a,b\in K$, $g\in K^\times$ be such that $a-b\preceq g$. Then
$$  \ddeg P_{+a,\times g}\ =\ \ddeg P_{+b,\times g}.$$
\end{lemma}

\begin{lemma} If $g,h\in K^\times$ and $g\preceq h$, then
$\ddeg P_{\times g}\le \ddeg P_{\times h}$.
\end{lemma}

\noindent
For these facts, see \cite[Lemma~6.6.5(i), Corollary~6.6.6,  Corollary~6.6.7]{ADH}.

\section{Eventual Behavior}\label{evtbeh}

\noindent
{\em In this section $\Gamma\ne \{0\}$.
We let $\phi$ range over~$K^\times$, and $\bsigma$, $\btau$ over $\N^*$. We also 
fix a differential polynomial $P\in K\{Y\}^{\ne}$}. Here we generalize parts of \cite[Sections~11.1, 11.2]{ADH} by dropping the assumption
there that $K$ is asymptotic. 
The condition $v\phi < (\Gamma^{>})'$ there becomes 
the condition $v\phi\in \Gamma(\der)$ here.

\subsection*{Behavior of $v F^{n}_{k}(\phi)$} The differential
polynomials $F^n_k(X)\in \Q\{X\}\subseteq K\{X\}$ for $0\le k \le n$ were introduced in
\cite[Section 5.7]{ADH} in connection with compositional conjugation: there we considered the
$K$-algebra morphism
$$Q\mapsto Q^\phi\colon K\{Y\}\to K^\phi\{Y\}$$ defined by requiring that $Q(y)=Q^\phi(y)$ for 
$Q\in K\{Y\}$ and all $y$ in all differential field extensions of $K$. The $F^n_k(X)$ ($1\leq k\leq n$) satisfy
$$(Y^{(n)})^\phi = F^n_n(\phi)Y^{(n)}+F^n_{n-1}(\phi)Y^{(n-1)}+\cdots+F_1^n(\phi)Y'$$
and $F^0_0=1$, $F^n_0=0$ for $n\geq 1$. (For example, $F^1_1=X$ and $F^2_2=X^2$, $F^2_1=X'$.)
We also recall from there that
for $\btau=\tau_1\cdots\tau_d\ge \bsigma=\sigma_1\cdots\sigma_d$,
$$F^{\btau}_{\bsigma}\ :=\ F^{\tau_1}_{\sigma_1}\cdots F^{\tau_d}_{\sigma_d}.$$ 
In order to better understand $v(P^{\phi})$ as a function of 
$\phi$
we use from Lemma~5.7.4 in~\cite{ADH} and its proof the identities 
\begin{equation}\label{eq:coeffs of Pphi}
(Y^{[\btau]})^\phi\ =\ \sum_{\bsigma\le \btau}F_{\bsigma}^{\btau}(\phi)Y^{[\bsigma]}, \qquad
(P^\phi)_{[\bsigma]}\  =\ \sum_{\btau \geq \bsigma} F^{\btau}_{\bsigma}(\phi) P_{[\btau]}. 
\end{equation} 
The next two lemmas have the same proof as \cite[Lemmas 11.1.1, 11.1.2]{ADH}.

\begin{lemma}\label{compconjval, general}
If $\der\mathcal{O}\subseteq\mathcal{O}$ and $\phi\preceq 1$, then $v(P^\phi)\geq v(P)$, with equality if $\phi\asymp 1$.
\end{lemma}

\noindent
We set $\derdelta=\phi^{-1}\der$ in the next two results.

\begin{lemma}\label{compconjval} 
Suppose that $\derdelta\smallo\subseteq \smallo$, and let 
$0\leq k\leq n$.
\begin{enumerate}
\item[\textup{(i)}] If $\phi^\dagger\preceq \phi$, then $F^n_k(\phi)\preceq \phi^n$ and $F^n_n(\phi)=\phi^n$.
\item[\textup{(ii)}] If $\phi^\dagger\prec \phi$ and $k<n$, then 
$F^n_k(\phi)\prec\phi^n$.
\end{enumerate}
\end{lemma}

\begin{cor}\label{compconjval, cor 2}
Suppose that $\derdelta\smallo\subseteq 
\smallo$ and $\phi^\dagger\preceq \phi$, and  $\btau\geq\bsigma$. Then
$F^\btau_{\bsigma}(\phi)\preceq
\phi^{\|\btau\|}$ and $F^\btau_{\btau}(\phi)=\phi^{\|\btau\|}$. If $\phi^\dagger \prec \phi$ and $\btau >\bsigma$, then $F^\btau_{\bsigma}(\phi)\prec
\phi^{\|\btau\|}$.
\end{cor}

\noindent
Let $P$ have order $\le r$, so $P=\sum_{\i}P_{\i}Y^{\i}$ with $\i$ ranging over $\N^{1+r}$. We define the {\it dominant degree}\/ $\ddeg P\in \N$ and the {\it dominant weight}\/ $\dwt P\in \N$ by
$$\ddeg P= \max\big\{|\i|:\ P_{\i}\asymp P\big\}, \qquad 
\dwt P=\max\big\{\|\i\|:\ P_{\i}\asymp P\big\}.$$
Thus if $K$ has small derivation, then $\ddeg P=\deg D_P$ as in the previous section, and $\dwt P=\wt D_P$, agreeing with the dominant weight from \cite[Sections~4.5, 6.6]{ADH}.  

\begin{lemma}\label{extra} Suppose $\der\mathcal{O}\subseteq \mathcal{O}$ and
$\phi\asymp 1$. Then $\ddeg P^{\phi}=\ddeg P$.
\end{lemma}
\begin{proof} Set 
$$d:= \ddeg P, \quad I_d:= \big\{\i:\ P_{\i}\asymp P,\ |\i|=d\big\}, \quad
I_{<d}:=\big\{\i:\ P_{\i}\asymp P,\ |\i|<d\big\}.$$ 
Then
$$P=Q+R+S\qquad\text{ with $Q:=\sum_{\i\in I_d} P_{\i}Y^{\i}$,\quad $R:=\sum_{\i\in I_{<d}} P_{\i}Y^{\i}$,}$$
and so 
$$P^\phi=Q^\phi + R^\phi + S^\phi, \quad P^\phi\asymp P,\quad Q^\phi\asymp Q \asymp P,\quad R^\phi\asymp R,\quad S^\phi\asymp S\prec P,$$
by Lemma~\ref{compconjval, general}, and $R\asymp P$ if $R\ne 0$. Also $\deg Q^\phi=\deg Q =d$ and $\deg R^\phi=\deg R< d$ by \cite[Corollary 5.7.5]{ADH}, and thus $\ddeg P^\phi=d$.
\end{proof}

\noindent
It is convenient to
introduce two operators $\operatorname{D}, \operatorname{W}\colon K\{Y\}^{\ne} \to K\{Y\}^{\ne}$:
\begin{align*} \operatorname{D}(P)\ &:=\ \sum_{\i\in I} P_{\i}Y^{\i}, \qquad 
I\ :=\ \{\i:\ P_{\i}\asymp P\},\\
    \operatorname{W}(P)\ &:=\ \sum_{\i\in J}P_{\i}Y^{\i}, \qquad 
    J\ :=\ \big\{\i\in I:\  \|\i\|=\dwt P\big\}.
\end{align*}
Thus $\operatorname{D}(P)$ and $\operatorname{W}(P)$ are of degree $\ddeg P$, and every monomial $Y^{\i}$ occurring 
in~$\operatorname{W}(P)$ has weight $\|\i\|=\dwt P$. Note that
$P\asymp \operatorname{D}(P)\asymp \operatorname{W}(P)$.  
If $K$ has small derivation, then the nonzero coefficients of $\operatorname{D}(P)$ are~$\asymp\mathfrak d_P$, and the image of~$\mathfrak d_P^{-1}\operatorname{D}(P)$ under the natural
differential ring morphism $\mathcal O\{Y\}\to\res(K)\{Y\}$ equals the dominant part $D_P$ of $P$.

\begin{lemma}\label{dconst} Suppose $\Gamma^{>}$ has no smallest element and
$\der \smallo\subseteq \smallo$. Then there exists an ${\alpha\in \Gamma^{<}}$ such that for $w:=\dwt P$ we have
 $$\operatorname{D}(P^\phi)\ \sim\ \phi^w\operatorname{W}(P)$$ 
for all $\phi$ with $\alpha < v\phi < 0$, so
$\ddeg P^{\phi}=\ddeg P$ and $\dwt P^{\phi}=\dwt P$ 
for such $\phi$.  
\end{lemma}
\begin{proof} For any monomial $Y^{\i}=Y^{[\btau]}$ we have $(Y^{[\btau]})^\phi=\sum_{\bsigma\le \btau}F_{\bsigma}^{\btau}(\phi)Y^{[\bsigma]}$ by \eqref{eq:coeffs of Pphi}.
Now let $\phi\succ 1$. Then $\phi^\dagger \prec \phi$: this is clear if $\phi'\preceq \phi$, and follows from \cite[Lem\-ma~6.4.1(iii)]{ADH} when $\phi'\succ \phi$. 
Thus by Corollary~\ref{compconjval, cor 2} and using $\|\i\|=\|\btau\|$:
$$(Y^{\i})^\phi\ \sim\ \phi^{\|\i\|}Y^{\i}.$$ 
Now $P=\operatorname{W}(P)+ Q$ with $Q\in K\{Y\}$, and for each monomial $Y^{\i}$, either
$Q_{\i}\prec P$, or $Q_{\i}=P_{\i}\asymp P$ and $\|\i\|< \dwt P$. Then
$$P^\phi\ =\ \operatorname{W}(P)^\phi + Q^\phi,\qquad 
\operatorname{W}(P)^\phi\ \sim\ \phi^w \operatorname{W}(P) \text{ for $w:=\dwt P$.}$$ 
Now $\Gamma^>$ has no smallest element, so given any $\beta\in\Gamma^>$ and $n\geq 1$ there is
an~$\alpha\in\Gamma^>$ such that $n\gamma<\beta$ whenever $\gamma\in\Gamma$ and $0<\gamma<\alpha$.
Thus by considering the individual monomials in~$Q$ we obtain an $\alpha\in \Gamma^{<}$ such that
$Q^\phi\prec \phi^w \operatorname{W}(P)$ whenever $\alpha<v\phi<0$. Any such $\alpha$ witnesses the property stated in the lemma.
\end{proof}
    
\begin{cor}\label{cordconst} If $\Gamma^{>}$ has no least element, 
$\phi_0\in K^\times$ and $v(\phi_0)\in \Gamma(\der)$, then
there exists $\alpha<v(\phi_0)$ such that
$\ddeg P^{\phi_0}=\ddeg P^{\phi}$ whenever
$\alpha < v(\phi) < v(\phi_0)$. 
\end{cor}
\begin{proof} Apply Lemma~\ref{dconst} to $K^{\phi_0}$ and $P^{\phi_0}$ in the role of $K$ and $P$.
\end{proof}

\subsection*{Newton degree} 
{\em In this subsection 
we assume $\Gamma^{>}$ has no least element}. 
Let $P\in K\{Y\}^{\ne}$ have order 
$\le r\in \N$. For $d\le \deg P$ we define
$$\Gamma(P,d)\ :=\ \big\{\gamma\in \Gamma(\der):\ \text{$\ddeg P^{\phi}=d $ for some $\phi$ with $v\phi=\gamma$}\big\}.$$
Note that in this definition of $\Gamma(P,d)$ we can replace
``some'' by ``all'' in view of Lemma~\ref{extra}, and hence the nonempty sets
among the $\Gamma(P,d)$ with $d\leq\deg P$ partition $\Gamma(\der)$. 
Note also that if $\gamma\in \Gamma(P,d)$, then 
$(\gamma-\alpha,\gamma]\subseteq \Gamma(P,d)$ for some $\alpha\in \Gamma^{>}$ by Corollary~\ref{cordconst}, so
each convex component of $\Gamma(P,d)$ in $\Gamma$ is infinite.

\begin{lemma}\label{convex} The set $\Gamma(P,d)$ has only finitely many convex components in $\Gamma$. 
\end{lemma}
\begin{proof} Let $\i$ range over the tuples $(i_0,\dots, i_r)\in \N^{1+r}$ with $|\i|\le \deg P$, and likewise for $\j$. Let
$N$ be the number of pairs $(\i, \j)$ with $\i\ne \j$.
We claim that for every~${\phi_0\in K^{\times}}$ with $v\phi_0\in \Gamma(\der)$ the set $\Gamma(P,d)$ has at most $N+1$ convex components with an element $\le v\phi_0$. (It follows easily from this claim that $\Gamma(P,d)$ has at most~$N+1$ convex components.) By renaming
$K^{\phi_0}$ and $P^{\phi_0}$ as $K$ and $P$ 
it suffices to prove the claim for $\phi_0=1$.  So we assume that
$\der\smallo\subseteq \smallo$ and have to show that $\Gamma(P,d)$ has at most $N+1$ components with an element $\le 0$.
We now restrict $\i$ further by the requirement that $P_{\i}\ne 0$, and likewise for $\j$. By the proof of Lem\-ma~\ref{dconst},
$$\ddeg P^\phi\ =\ \max\left\{|\i|:\ vP_{\i}+\|\i\|v\phi=\min_{\j}vP_{\j}+\|\j\|v\phi\right\}\ \text{ for $v\phi<0$.}$$
For each $\i$ we have the function $f_{\i}: \Q\Gamma\to \Q\Gamma$ given by $f_{\i}(\gamma)=vP_{\i}+\|\i\|\gamma$.
For any $\i$,~$\j$, either $f_{\i}=f_{\j}$ or we have a unique
$\gamma=\gamma_{\i,\j}\in \Q\Gamma$ with $f_{\i}(\gamma)=f_{\j}(\gamma)$. Let $\gamma_1< \dots < \gamma_M$ with $M\le N$ be the distinct values of $\gamma_{\i,\j}<0$ obtained in this way, and set $\gamma_0:= -\infty$ and $\gamma_{M+1}:= 0$. Then on each interval $(\gamma_m, \gamma_{m+1})$ with $0\le m \le M$, the
functions $f_{\i}-f_{\j}$ have constant sign: $-$, $0$, or $+$.
In view of the above identity for $\ddeg P^{\phi}$ it follows easily that for each 
$m$ with $0\le m\le M$ the value of $\ddeg P^\phi$ is constant
as $v\phi$ ranges over $(\gamma_m, \gamma_{m+1})\cap \Gamma$.
Thus $\Gamma(P,d)$ has at most $M+1$ convex components.
\end{proof}

\noindent
It follows from Lemmas~\ref{dconst} and~\ref{convex} that 
there exists 
$d\le \deg P$ and a $\phi_0\in K^\times$ such that 
$v\phi_0\in \Gamma(\der)$, $v\phi_0$ is not
maximal in $\Gamma(\der)$, and $\ddeg P^\phi=d$ for all $\phi\preceq \phi_0$
with $v\phi\in \Gamma(\der)$. 
We now define the {\it Newton degree}\/ $\ndeg P$ of $P$ to be this eventual value $d\in \N$ of $\ddeg P^\phi$. Note that if $\Gamma(\der)$
does have a maximal element $v\phi$, then $$\ndeg P\ =\ \ddeg P^\phi.$$ 
Also, for $f\in K^\times$ and  $Q\in K\{Y\}^{\ne}$ we have
$$\ndeg P^f\ =\ \ndeg P, \qquad \ndeg PQ\ =\ \ndeg P + \ndeg Q.$$ 

\subsection*{Newton degree and multiplicative conjugation}
{\it In this subsection $\Gamma^{>}$ has no least element.}\/ Here we
consider the behavior of $\ndeg P_{\times g}$ as a function of~$g\in K^\times$.
Indeed, $\ndeg P_{\times g}\ge 1$ is a useful necessary
condition for the existence of a
zero~$f\preceq g$ of $P$ in a strict extension of $K$, as stated in
the following generalization of \cite[Lem\-ma~11.2.1]{ADH}:

\begin{lemma}\label{newtonzero}
Let $g\in K^\times$ and suppose some $f\preceq g$ in a strict extension of 
$K$ satisfies $P(f)=0$. Then $\ndeg P_{\times g} \geq 1$.
\end{lemma}
\begin{proof} For such $f$ 
we have $f=ag$ with $a\preceq 1$, and $Q(a)=0$ for $Q:=P_{\times g}$. So  $Q^\phi(a)=0$ for all 
$\phi$ with $v\phi\in\Gamma(\der)$, hence 
$\ddeg Q^\phi\geq 1$ for those $\phi$, and thus $\ndeg Q\geq 1$.
\end{proof}

\noindent
Next some results on {\em Newton degree\/} that follow easily 
from corresponding facts at the end of Section~\ref{sec:prelim} on {\em dominant degree}, using also that compositional conjugation commutes with additive and 
multiplicative conjugation by \cite[Lemma~5.7.1]{ADH}.

\begin{lemma} If $a\in K$ and $a\preceq 1$, then $\ndeg {P_{+a}}=\ndeg{P}$.
\end{lemma}

\begin{lemma} \label{aftranew} Let $a,b\in K$, $g\in K^\times$ 
be such that $a-b\preceq g$. Then 
$$\ndeg P_{+a,\times g}\ =\ \ndeg P_{+b,\times g}.$$
\end{lemma}

\begin{lemma}\label{ndegdecr} If $g,h\in K^\times$ and $g\preceq h$, then
$\ndeg P_{\times g}\le \ndeg P_{\times h}$.
\end{lemma}

\noindent
For $g\in K^\times$ we set $\ndeg_{\prec g}P\ :=\ \max\{\ndeg P_{\times f}:\ f\prec g\}$.

\begin{lemma}\label{ndeg} For $a,g\in K$ with $a\prec g$ we have
$\ndeg_{\prec g} P_{+a}=\ndeg_{\prec g}P$.
\end{lemma}
\begin{proof} Use that $\ndeg P_{+a,\times f}=
\ndeg P_{\times f}$ for $a\preceq f\prec g$, by Lemma~\ref{aftranew}.
\end{proof}

\noindent
It will also be convenient to define for $\gamma\in \Gamma$,
 $$\ndeg_{\geq \gamma} P\ :=\ \max\{\ndeg P_{\times g}:\  g\in K^\times, vg\ge \gamma\}.$$
By Lemma~\ref{ndegdecr}, $\ndeg_{\geq\gamma} P=\ndeg P_{\times g}$ for any $g\in K^\times$ with $\gamma=vg$.
From Lemmas~\ref{ndegdecr} and \ref{aftranew} we easily obtain: 

\begin{cor}\label{cor:ndegdecr}
Let $a,b\in K$ and $\alpha,\beta\in\Gamma$ be such that $v(b-a)\geq\alpha$ and $\beta\geq\alpha$. Then $\ndeg_{\geq\beta} P_{+b} \leq \ndeg_{\geq\alpha} P_{+a}$.
\end{cor}

\subsection*{Newton degree in a cut}
{\it In this subsection $\Gamma^{>}$ has no least element.}\/
We do not need the material here to obtain the main theorem. It
is only used in proving Corollaries~\ref{ppkell} and \ref{dpkell}, which 
are of interest for other reasons. 
 
Let $(a_{\rho})$ be a pc-sequence in~$K$, and
put $\gamma_{\rho}=v(a_{s(\rho)}-a_{\rho})\in \Gamma_{\infty}$, where $s(\rho)$ is the immediate successor of $\rho$.   
Using Corollary~\ref{cor:ndegdecr} in place of
\cite[Corollary~11.2.8]{ADH} we generalize \cite[Lemma~11.2.11]{ADH}:

\begin{lemma}\label{ndegeq}  
There is an index $\rho_0$ and $d\in \N$ such that for all $\rho > \rho_0$ 
we have
$\gamma_{\rho}\in \Gamma$ and $\ndeg_{\geq\gamma_{\rho}}P_{+a_{\rho}} = d$.
Denoting this number $d$ by $d\big(P,(a_{\rho})\big)$, we have $d\big(P,(a_{\rho})\big)=d\big(P,(b_{\sigma})\big)$
whenever $(b_{\sigma})$ is a pc-sequence in~$K$ equivalent to $(a_{\rho})$.
\end{lemma}

\noindent
As in \cite[Section~11.2]{ADH},
we now associate to each pc-sequence $(a_{\rho})$ in $K$ an object~$c_K(a_{\rho})$,
the {\it cut}\/ defined by~$(a_{\rho})$ in $K$, such that if $(b_{\sigma})$ is 
also a pc-sequence in~$K$, then 
$$c_K(a_{\rho})=c_K(b_{\sigma}) \ \Longleftrightarrow\ \text{$(a_{\rho})$ and $(b_{\sigma})$ are equivalent.}$$
We do this in such a way that the cuts $c_K(a_{\rho})$, with $(a_{\rho})$ a
pc-sequence in $K$, are the elements of a set $c(K)$.  Using Lem\-ma~\ref{ndegeq} we define for ${\bf a}\in c(K)$ the {\it Newton degree of $P$ in the cut $\bf a$}\/ as
$$\ndeg_{\bf{a}}P\ :=\ d\big(P,(a_\rho)\big)\ =\ \text{eventual value of $\ndeg_{\geq\gamma_{\rho}}P_{+a_{\rho}}$,}$$
where $(a_{\rho})$ is any pc-sequence in $K$ with ${\bf a}=c_K(a_{\rho})$. 
Let~$(a_{\rho})$ be a pc-sequence in~$K$ and ${\bf a}=c_K(a_{\rho})$.
For $y\in K$ the cut $c_K(a_{\rho} + y)$ depends only on~$({\bf a}, y)$, and so we can set ${\bf a}+y:=c_K(a_\rho+y)$. Likewise, for $y\in K^\times$ the cut~$c_K(a_{\rho}y)$ depends only on $({\bf a}, y)$, and so we can set 
${\bf a}\cdot y:=c_K(a_\rho y)$. We record some basic facts about~$\ndeg_{\bf a} P$: 

\begin{lemma} \label{lem:basic facts on deg_a}
Let $(a_{\rho})$ be a pc-sequence in $K$, ${\bf a}=c_K(a_\rho)$. Then
\begin{enumerate}
\item[\textup{(i)}] $\ndeg_{\bf a} P \leq \deg P$;
\item[\textup{(ii)}] $\ndeg_{\bf a} P^f = \ndeg_{\bf a} P$ for   $f\in K^\times$;
\item[\textup{(iii)}] $\ndeg_{\bf a} P_{+y} = \ndeg_{{\bf a}+y} P$  for $y\in K$;
\item[\textup{(iv)}] if $y\in K$ and $vy$ is in the width of $(a_\rho)$, then $\ndeg_{\bf a} P_{+y} = \ndeg_{\bf a} P$;
\item[\textup{(v)}] $\ndeg_{\bf a} P_{\times y} = \ndeg_{{\bf a}\cdot y} P$ for $y\in K^\times$;
\item[\textup{(vi)}] if $Q\in K\{Y\}^{\neq}$, then $\ndeg_{\bf a} PQ = \ndeg_{\bf a} P + \ndeg_{\bf a} Q$;
\item[\textup{(vii)}] if $P(\ell)=0$ for some pseudolimit $\ell$ of $(a_{\rho})$ in a strict   extension of~$K$, then $\ndeg_{\bf a} P \ge 1$;
\end{enumerate}
\end{lemma}

\begin{proof} Most of these items are routine or follow easily from earlier facts.   
Item (iv) follows from (iii), and (vii) from  Lemma~\ref{newtonzero}. 
\end{proof}

\section{Flexibility}\label{secflex}

\noindent
We assume in this section about our valued differential field $K$
that $$\Gamma\ne \{0\}, \qquad\der\ne 0.$$ 
After the first three lemmas 
we introduce the useful condition of {\em flexibility\/}, which plays a key role in the rest of the story.  

\begin{lemma}
Let $P\in K\{Y\}^{\ne}$ be such that
$\deg P\ge 1$. Suppose $\der$ is small and the derivation of $\res(K)$ is nontrivial.
Then the set $$\big\{vP(y):\ y\in K,\ P(y)\neq 0\big\}\ \subseteq\ \Gamma$$
is coinitial in $\Gamma$.
\end{lemma}
\begin{proof}
Given $Q\in K\{Y\}$, the gaussian valuation~$v(Q_{\times f})$ of $Q_{\times f}$ for $f\in K$ depends only on $v(f)$ by \cite[Lemma~4.5.1(ii)]{ADH}, and so we obtain a function $v_Q\colon \Gamma_\infty\to\Gamma_\infty$ with $v_Q(vf)=v(Q_{\times f})$ for $f\in K$.  We have
$v_P(\gamma)=\min_d v_{P_d}(\gamma)\in \Gamma$ for $\gamma\in\Gamma$, and by~\cite[Corollary~6.1.3]{ADH}, $v_{P_d}(\gamma)=v(P_d)+d\gamma+o(\gamma)$ if $P_d\neq 0$ and  $\gamma\in \Gamma^{\ne}$.
Using also $\deg P\ge 1$, it follows that $v_P(\Gamma)$ is a coinitial subset of $\Gamma$. 
By \cite[Lemma~4.5.2]{ADH}, there is for each $\beta\in v_P(\Gamma)$ a $y\in K$ with $vP(y)=\beta$.
\end{proof}

\begin{lemma}\label{flex} Let $P\in K\{Y\}^{\ne}$ be such that
$\deg P\ge 1$. Then
the set $$\big\{vP(y):\ y\in K\big\}\ \subseteq\ \Gamma_{\infty}$$ is infinite.
\end{lemma}
\begin{proof} By compositional conjugation we arrange that
$\der$ is small. Take an elementary extension $L$ of $K$ such that
$\Gamma_L$ contains an element $>\Gamma$. Let $\Delta$ be the convex hull of~$\Gamma$ in $\Gamma_L$, and let $L_{\Delta}$ be the
$\Delta$-coarsening of $L$ with valuation $\dot{v}$ and
(nontrivial) value group $\dot{\Gamma}_L=\Gamma_L/\Delta$. 
By Lemma~\ref{lem:4.4.4}, the derivation of $L_\Delta$ remains small, and since $\der\ne 0$, the derivation of $\res(L_{\Delta})$ is nontrivial.
So by  the preceding lemma, the set
$\big\{\dot{v}P(y):\ y\in L,\ P(y)\ne 0\big\}$ is coinitial in 
$\dot{\Gamma}_L$.
Hence the set $\big\{vP(y):\ y\in L,\ P(y)\ne 0\big\}$ is coinitial in $\Gamma_L$. Thus
$\big\{vP(y):\ y\in K,\ P(y)\ne 0\big\}$ is coinitial in $\Gamma$, and  hence infinite. 
\end{proof}

\begin{lemma}\label{flex0} Suppose $\Gamma^{>}$ has no least element and $S(\der)=\{0\}$. Let $P\in K\{Y\}^{\ne}$ be such that 
$\ndeg P\ge 1$, 
and let $\beta\in \Gamma^{>}$. Then the set
$$\big\{vP(y):\ y\in K,\ |vy|<\beta\big\}\ \subseteq\ \Gamma_{\infty}$$
is infinite. 
\end{lemma}
\begin{proof} Let $\gamma\in \Gamma(\der)$ and
$\delta\in \Gamma\setminus \Gamma(\der)$; then there are $a,g\in K$ such that 
$$a\prec 1, \quad vg=\gamma, \quad 0\ <\  v(g^{-1}a')\ \le\ \delta-\gamma.$$
To see this, take $a,d\in K$ such that $a\prec 1$, $vd=\delta$, and $d^{-1}a'\succeq 1$. Take $g\in K$ with $vg=\gamma$. Then $a'\succeq d$, and
so $g^{-1}a'\succeq g^{-1}d$. It remains to note that $g^{-1}a'\prec 1$. 
 
 This fact and Lemma~\ref{epsgamdel} yield an elementary extension $L$ of $K$, with
 elements~$\phi\in L^\times$ and $a\in \smallo_L$ such that $v\phi\in \Gamma_L(\der_L)$, $v\phi\ge \Gamma(\der)$ and
$0<v(\phi^{-1} a')<\Gamma^{>}$. Let $\Delta$ be the convex subgroup
of $\Gamma_L$ consisting of the $\epsilon\in \Gamma_L$ with
$|\epsilon|< \Gamma^{>}$. Then $\res(L^{\phi}_{\Delta})$ has nontrivial derivation with value group $\Delta\ne \{0\}$.
Take a nonzero $f\in L$ such that $f^{-1}P^\phi\asymp 1$ in
$L^{\phi}\{Y\}$. Let $P_{\Delta}\in \res(L^{\phi}_{\Delta})\{Y\}$
be the image of $f^{-1}P^\phi\in \mathcal{O}_{L^\phi}\{Y\}$
under the natural map $\mathcal{O}_{L^\phi}\{Y\}\to \res(L^{\phi}_{\Delta})\{Y\}$. From $\ndeg P\ge 1$ it follows that
$\deg P_{\Delta}\ge 1$. Now apply Lemma~\ref{flex} to 
$\res(L^{\phi}_{\Delta})$ and~$P_{\Delta}$ in the role of~$K$ and~$P$.  
\end{proof}

\noindent
Recall that $a^\sim$ is the equivalence class of $a\in K^\times$ with respect to the equivalence relation $\sim$ on $K^\times$. We define $K$ to be {\it flexible\/} if $\Gamma^{>}$ has no least element and for all $P\in K\{Y\}^{\ne}$ with
$\ndeg P\ge 1$ and all $\beta\in \Gamma^{>}$ the set
$$\big\{P(y)^{\sim}:\ y\in K,\ |vy|<\beta,\ P(y)\ne 0\big\}$$
is infinite.
Flexibility is an elementary condition on valued 
differential fields, in the sense of being 
expressible by a set of sentences
in the natural first-order language for these structures. 
Flexibility is invariant under compositional conjugation. By Lemma~\ref{flex0} we have:

\begin{cor} If $\Gamma^{>}$ has no least element and $S(\der)=\{0\}$, then $K$ is flexible.
\end{cor}
  
\noindent
Combined with
earlier results on $S(\der)$ this gives large
classes of valued differential fields that are flexible. For example, if $\Gamma^{>}$ has no least element, then $K$ is flexible whenever $K$ is asymptotic or $\Gamma$ is archimedean.
If $\Gamma^{>}$ has no least element and $K$ is flexible, does it
follow that $S(\der)=\{0\}$? We don't know.  

\begin{remark} In \cite[p.~292]{ADHO} we defined a less ``flexible'' notion of flexibility. We stated there as 
Theorem~4.1, without proof, that every real closed $H$-field has a spherically complete immediate $H$-field extension, and mentioned that we used flexibility in handling the case
where the real closed $H$-field has no asymptotic integration. It turned out that for that case ``Theorem~4.1'' was not needed in \cite{ADH}, and so it was not included there. As we saw in the introduction,
Theorem~4.1 from \cite{ADHO} is now available, even without the {\em real closed\/} assumption, as a special case of the main theorem  of the
present paper. 
\end{remark}

\section{Lemmas on Flexible Valued Differential Fields} \label{sec:flexlemmas}

\noindent
{\em In this section we assume about $K$ that $\der\ne 0$, $\Gamma\ne \{0\}$, and
$\Gamma^{>}$ has no least element}. (Flexibility is only assumed in Lemmas~\ref{suplim4} and~\ref{Zw}.)
We let $a$, $b$, $y$ range over $K$
and $\fm$, $\fn$, $\fd$, $\fv$, $\fw$ over $K^\times$.
Also, $P$ and $Q$ range over $K\{Y\}^{\ne}$. 

Using strict extensions and flexibility we now adapt the subsection ``Vanishing'' of \cite[Section 11.4]{ADH} to our more general setting.   

\medskip\noindent
Let $\ell$ be an element in an extension $L$ of $K$ such that $\ell\notin K$ and $v(K-\ell):=\big\{v(a-\ell):\ a\in K\big\}$ has 
no largest element. Recall that then $\ell$ is a pseudolimit of a divergent 
pc-sequence in $K$ and
$v(K-\ell)\subseteq \Gamma$.

We say that {\it $P$ vanishes at $(K,\ell)$}\/ if for all $a$ and $\fv$ with $a-\ell \prec \fv$ we have 
$\ndeg_{\prec \fv} P_{+a} \ge 1$, that is $\ndeg P_{+a,\times b}\geq 1$ for some $b\prec\mathfrak v$. 
By Lemma~\ref{newtonzero}, if $L$ is an immediate strict extension of $K$ and $P(\ell)=0$, then $\ndeg P_{+a,\times b}\geq 1$ whenever $\ell-a\preceq b$, hence $P$ vanishes at $(K,\ell)$.
Let $Z(K,\ell)$ be the set of all $P$ that vanish at $(K,\ell)$. Here are some frequently used basic facts: \begin{enumerate}
\item $P\in Z(K,\ell)\ \Longleftrightarrow\ P_{+b}\in Z(K, \ell-b)$;
\item $P\in Z(K,\ell)\ \Longleftrightarrow\ P_{\times \fm}\in Z(K,\ell/\fm)$;
\item $P\in Z(K,\ell)\ \Longrightarrow\ PQ\in Z(K,\ell) \text{  for all $Q$}$;  
\item $P\in K^\times\ \Longrightarrow\ P\notin Z(K,\ell)$.
\end{enumerate}
(In general, $Z(K,\ell)\cup\{0\}$ is not closed under addition, see the remark following the proof of Corollary~\ref{ppkell} below.)
Moreover, if $P\notin Z(K,\ell)$, we have
$a$, $\fv$ with $a-\ell \prec \fv$ and $\ndeg_{\prec \fv}P_{+a} =0$, and then
also $\ndeg_{\prec \fv} P_{+b} =0$ for any $b$ with $b-\ell \prec \fv$,
by Lemma~\ref{ndeg}. 

\begin{lemma}\label{suplim1} $Y-b\notin Z(K,\ell)$.
\end{lemma}
\begin{proof} Take $a$ and $\fv$ such that $a-\ell \prec \fv \asymp b-\ell$.
Then for $P:= Y-b$ and $\fm \prec \fv$ we have $P_{+a, \times \fm}=\fm Y + (a-b)$ and $\fm\prec a-b$, so $D_{P_{+a,\times \fm}} \in \res(K)^\times$. It follows that
$\ndeg_{\prec \fv}P_{+a} =0$.
\end{proof}

\begin{lemma}\label{Z1} Suppose $P\notin Z(K,\ell)$, and let $a$,~$\fv$ be such that 
$a-\ell \prec \fv$ and
$\ndeg_{\prec \fv}P_{+a} =0$. Then $P(f)\sim P(a)$ for all 
$f$ in all strict extensions of~$K$ with 
$f-a\asymp \fm \prec \fv$ for some $\fm$. \textup{(}Recall: $\fm\in K^\times$ by convention.\textup{)}
\end{lemma}
\begin{proof} Let $f$ in a strict extension $E$ of $K$
satisfy $f-a\asymp \fm \prec \fv$, so
$f=a+\fm u$ with $u\asymp 1$ in $E$. Now 
$$P_{+a,\times \fm}\ =\ P(a) + R\qquad\text{ with $R\in K\{Y\}$, $R(0)=0$,}$$ 
so for $\phi\in K^\times$ we have
$$ P^\phi_{+a,\times\fm}\ =\ P(a) +R^{\phi}.$$
From $\ndeg P_{+a,\times \fm}=0$ we get $\phi\in K^\times$ with
 $\der\smallo\subseteq \phi\smallo$ and $R^{\phi}\prec P(a)$.
Thus 
$$P(f)\ =\ P_{+a,\times\fm}(u)\ =\ P^\phi_{+a,\times\fm}(u)\ =\ P(a) + R^{\phi}(u)\quad \text{in $E^{\phi}$,}$$ with $R^{\phi}(u)\preceq R^{\phi}\prec P(a)$ in $E^{\phi}$, so $P(f)\sim P(a)$.  
\end{proof}

\noindent
Suppose $L$ is a strict extension of $K$. Then the conclusion applies to $f=\ell$, and so for $P$ and $a$, $\fv$ as in the lemma we have $P(\ell) \sim P(a)$, hence $P(\ell)\ne 0$. Thus for $P$, $a$, $\fv$ as in the lemma we have $P(f) \sim P(a) \sim P(\ell)$ for all $f\in K$ with $f-\ell \prec \fv$. 

\begin{lemma}\label{suplim3} Suppose that $P,Q\notin Z(K,\ell)$. Then $PQ\notin Z(K,\ell)$.
\end{lemma}
\begin{proof} Take $a$, $b$, $\fv$, $\fw$ such that $a-\ell \prec \fv$, $b-\ell \prec \fw$ and
$$ \ndeg_{\prec \fv} P_{+a}\ =\ \ndeg_{\prec \fw} Q_{+b}\ =\ 0.$$
We can assume $a-\ell \preceq b-\ell$. Take $\fn\asymp a-\ell$ and $d\in K$
with $d-\ell \prec \fn$. Then $d-\ell \prec \fv$ and $d-\ell \prec \fw$, so
$\ndeg_{\prec \fv}P_{+a} = \ndeg_{\prec \fv}P_{+d}=0$, and so $\ndeg_{\prec\fn}P_{+d}=0$. 
Likewise, $\ndeg_{\prec \fn}Q_{+d}=0$, so
$\ndeg_{\prec \fn}(PQ)_{+d}=0$.
\end{proof}

\begin{lemma}\label{suplim4} Assume $K$ is flexible. Let $P\in Z(K,\ell)$, and 
let any $b$ be given. Then there exists an $a$ such that $a-\ell \prec b-\ell$ and $P(a)\ne 0$, $P(a)\not\sim P(b)$. 
\end{lemma}
\begin{proof} Take $\fv\asymp b-\ell$ and $a_1\in K$ with $a_1-\ell \prec \fv$, so
$\ndeg_{\prec \fv}P_{+a_1} \ge 1$, which gives $\fm\prec \fv$ with 
$\ndeg P_{+a_1,\times \fm} \ge 1$. 
By flexibility of $K$,
the set $$\big\{P(a_1+\fm y)^{\sim}:\  |vy|<\beta,\ P(a_1+\fm y)\ne 0 \big\}$$ is infinite,  for each $\beta\in \Gamma^{>}$,
so we can take $y$ such that $a_1+\fm y-\ell \prec \fv$ and 
$0\ne P(a_1+\fm y)\not\sim P(b)$. Then
$a:= a_1+\fm y$ has the desired property. 
\end{proof}

\begin{lemma}\label{Zw} Assume $K$ is flexible, $L$ is a strict extension of $K$, $P,Q\notin Z(K,\ell)$ and $P-Q\in Z(K,\ell)$.
Then $P(\ell)\sim Q(\ell)$.
\end{lemma}
\begin{proof} By Lemma~\ref{suplim3} we have $b$ and $\fv$ such that
$$\ell -b \prec \fv, \quad \ndeg_{\prec \fv}P_{+b}=\ndeg_{\prec \fv}Q_{+b}=0.$$
Replacing $\ell$ by $\ell-b$ and $P, Q$ by $P_{+b}, Q_{+b}$ we arrange $b=0$, that is,
$$\ell \prec \fv, \quad \ndeg_{\prec \fv}P=\ndeg_{\prec \fv}Q=0, $$
so $P(0)\ne 0$ and $Q(0)\ne 0$. If $a\prec \fv$, then
by the remark preceding Lemma~\ref{suplim3},
$$P(a)\sim P(0)\sim P(\ell), \qquad Q(a) \sim Q(0) \sim Q(\ell).$$
If $P(\ell)\not\sim Q(\ell)$, then $P(0)\not\sim Q(0)$, so $(P-Q)(a)\sim (P-Q)(0)$
for all $a\prec \fv$, contradicting $P-Q\in Z(K, \ell)$ by Lemma~\ref{suplim4}.
Thus $P(\ell)\sim Q(\ell)$. 
\end{proof}

\subsection*{Relation to the Newton degree in a cut}
Let $(a_\rho)$ be a divergent pc-sequence in $K$ with pseudolimit $\ell$.
The following generalizes \cite[Lemma~11.4.11]{ADH}, with the same proof
except for using
Lemma~\ref{Z1} instead of \cite[Lemma~11.4.3]{ADH}. 

\begin{cor}\label{ppkell} If $P(a_{\rho}) \leadsto 0$, then $P\in Z(K,\ell)$.
\end{cor}
\begin{proof} Suppose $P\notin Z(K,\ell)$. Take $a$ and $\fv$ such that $a-\ell\prec \fv$ and $\ndeg_{\prec \fv}P_{+a}=0$. Now  $v(a-a_{\rho})=v(a-\ell)$, eventually, so by Lemma~\ref{Z1} we have $P(a_{\rho}) \sim P(a)$ eventually, so $v\big(P(a_{\rho})\big)= v\big(P(a)\big)\ne \infty$ eventually. 
\end{proof}

\noindent
In particular, if $P(a_{\rho}) \leadsto 0$,
then $P(Y)+\varepsilon\in Z(K,\ell)$ for all  $\varepsilon\in K$ such that $\varepsilon\prec P(a_\rho)$ eventually.
We now connect the notion of $P$ vanishing at $(K,\ell)$ with the Newton degree~$\ndeg_{\bf a} P$ of $P$ in the cut ${\bf a}=c_K(a_\rho)$ (introduced in the last subsection of Section~\ref{evtbeh}) generalizing \cite[Lemma~11.4.12]{ADH}. The proof is the same, except for using Lemma~\ref{ndeg} and Corollary~\ref{cor:ndegdecr} above 
instead of
\cite[Lemma~11.2.7]{ADH}:

\begin{cor}\label{dpkell} $\ndeg_{\bf a} P \ge 1\ \Longleftrightarrow\ 
P\in Z(K,\ell)$. More precisely,
$$\ndeg_{\bf a} P\  =\ 
\min\!\big\{\!\ndeg_{\prec \fv}P_{+a}:a-\ell\prec\fv\big\}.$$ 
\end{cor}

\begin{proof} We may assume $v(\ell- a_{\rho})$ is strictly increasing with $\rho$.  Given any index~$\rho$, take $\fv\asymp \ell -a_{\rho}$, take $\rho'>\rho$, and set $a:= a_{\rho'}$. Then $a-\ell \prec \fv$. Now $\gamma_{\rho} := v(\ell - a_{\rho})=v(\fv)=v(a-a_{\rho})$, and thus (using Corollary~\ref{cor:ndegdecr} for the last inequality):
$$\ndeg_{\prec \fv}P_{+a}\ \le\ \ndeg_{\preceq \fv}P_{+a}\  =\  \ndeg_{\geq \gamma_{\rho}}P_{+a}\ \le\ \ndeg_{\geq \gamma_{\rho}}P_{+a_{\rho}}.$$
It follows that $\min\!\big\{\!\ndeg_{\prec \fv}P_{+a}:a-\ell\prec\fv\big\}\leq \ndeg_{\bf a} P$. For the reverse inequality,  let $a$ and $\fv$ be such that $a-\ell \prec \fv$. Let $\rho$ be such that $\ell - a_{\rho}\preceq \ell-a$.  Then $a_{\rho}-a \prec \fv$ and  $\gamma_{\rho} = v(\ell -a_{\rho})> v(\fv)$, so by  Lemma~\ref{ndeg}:
$$\ndeg_{\geq\gamma_{\rho}}P_{+a_{\rho}}\ \le\  \ndeg_{\prec \fv}P_{+a_{\rho}}\ =\ \ndeg_{\prec \fv} P_{+a}.$$
Therefore $\ndeg_{\bf a} P\leq \min\!\big\{\!\ndeg_{\prec \fv}P_{+a}:a-\ell\prec\fv\big\}$.
\end{proof}

\section{Constructing Immediate Extensions} \label{sec:immext}

\noindent
Our goal in this section is to establish the following:

\begin{theorem}\label{thimm} Suppose $\der\ne 0$, $\Gamma\ne \{0\}$, $\Gamma^{>}$ has no least element, and $S(\der)=\{0\}$.  Then $K$ has the Krull property. 
\end{theorem}

\noindent
Much of this section is very similar to the subsection ``Constructing immediate extensions'' of  \cite[Section 11.4]{ADH}, but there are some differences that make it convenient to give 
all details.
In the next section we show how to derive our main theorem  from Theorem~\ref{thimm} by constructions involving coarsening by $S(\der)$.

\medskip\noindent
{\em In the rest of this section we assume about $K$ that $\der\ne 0$, $\Gamma\ne \{0\}$, and $\Gamma^{>}$ has no least element}. 
We also keep the notational 
conventions of the previous section, and assume that $\ell$ is an element of a {\em strict\/} extension $L$ of $K$.

\begin{lemma}\label{zdf} Suppose $Z(K,\ell)=\emptyset$. Then $P(\ell)\ne 0$ for all $P$, 
and $K\<\ell\>$ is an immediate strict extension of $K$.
Suppose also that $M$ is a strict extension of $K$ and $g\in M$ satisfies
$v(a-g)=v(a-\ell)$ for all $a$. Then there is a unique valued differential field
embedding $K\<\ell \> \to M$ over $K$ that sends $\ell$ to $g$. 
\end{lemma}
\begin{proof} Clearly $P(\ell)\ne 0$ for all $P$. Let any nonzero element $f=P(\ell)/Q(\ell)$ of the extension $K\<\ell\>$ of $K$ be given.
Lemma~\ref{suplim3} gives $a$ and $\fv$ such that
$$a-\ell\ \prec\ \fv,\qquad \dd_{\prec \fv} P_{+a}\ =\ \dd_{\prec \fv} Q_{+a}\ =\ 0,$$
and so $P(\ell) \sim P(a)$ and $Q(\ell) \sim Q(a)$ by Lemma~\ref{Z1}, and thus $f\sim P(a)/Q(a)$.
It follows that $K\<\ell\>$ is an immediate extension of $K$.

It is clear that $Z(K,g)=Z(K,\ell)=\emptyset$, so $g$ is differentially transcendental over~$K$ and
$K\<g\>$ is an immediate extension of $K$, by the first part of the proof.
Given any~$P$ we take $a$ and $\fv$ such that $a-\ell \prec \fv$ and
$\dd_{\prec \fv} P_{+a} =0$. Then $P(a)\sim P(g)$ and $P(a) \sim P(\ell)$, and
thus $vP(g)=vP(\ell)$.  Hence the unique differential field embedding $K\<\ell\> \to M$ over $K$ that sends $\ell$ to $g$ is also a valued field embedding. 
\end{proof}

\begin{lemma}\label{zda} Suppose $K$ is flexible, $Z(K,\ell) \ne \emptyset$, and $P$ is an element of
$Z(K,\ell)$ of minimal complexity. Then $K$ has an immediate strict extension $K\<f\>$ such that $P(f)=0$ and
$v(a-f)=v(a-\ell)$ for all $a$, and such that if $M$ is any
strict extension of $K$ and $s\in M$ satisfies $P(s)=0$ and
$v(a-s)=v(a-\ell)$ for all $a$, then there is a unique valued differential field
embedding $K\<f \> \to M$ over $K$ that sends $f$ to $s$.
\end{lemma}
\begin{proof} Let $P$ have order $r$ and take $p\in K[Y_0,\dots,Y_r]$ such that
$$P\ =\ p(Y, Y',\dots, Y^{(r)}).$$ 
Then $p$ is irreducible by $P$ having minimal complexity in $Z(K,\ell)$ and Lemma~\ref{suplim3}. Thus we have an integral domain 
$$K[y_0,\dots, y_r]\ =\ K[Y_0,\dots, Y_r]/(p), \qquad y_i=Y_i + (p) \text{ for $i=0,\dots,r$,}$$
with fraction field $K(y_0,\dots,y_r)=K(y_0,\dots,y_{r-1})[y_r]$ where $y_0,\dots,y_{r-1}$ are algebraically independent over $K$.
Let  $s\in K(y_0,\dots, y_r)^\times$, so $$s\ =\ g(y_0,\dots, y_r)/h(y_0,\dots, y_{r-1})$$ 
where $g\in K[Y_0,\dots,Y_r]^{\ne}$,
$h\in K[Y_0,\dots,Y_{r-1}]^{\neq}$, and $g(Y, Y',\dots, Y^{(r)})\notin Z(K,\ell)$. (This nonmembership in $Z(K,\ell)$ can be arranged by taking $g$ of lower degree in $Y_r$ than $p$.) 
The comment following the proof of
Lemma~\ref{Z1} gives an $a$ such that
$$g(\ell, \ell',\dots, \ell^{(r)})\ \sim\ g(a, a',\dots, a^{(r)}), \qquad h(\ell,\dots, \ell^{(r-1)})\ \sim\ h(a,\dots, a^{(r-1)}),$$
so $vg(\ell, \ell',\dots, \ell^{(r)}),\ vh(\ell,\dots, \ell^{(r-1)})\in \Gamma$.
We claim that 
$$vg(\ell, \ell',\dots, \ell^{(r)})-vh(\ell,\dots, \ell^{(r-1)})$$
depends only on $s$ and not on the choice of $g$ and $h$. To see this, let
$g_1\in K[Y_0,\dots, Y_r]$, $h_1\in K[Y_0,\dots, Y_{r-1}]$ be such that 
$g_1(Y,\dots, Y^{(r)})\notin Z(K,\ell)$, 
$h_1\ne 0$, and $s=g_1(y_0,\dots, y_r)/h_1(y_0,\dots, y_{r-1})$.
Then 
$$gh_1-g_1h \in pK[Y_0,\dots, Y_r], \qquad (gh_1)(Y,\dots, Y^{(r)}), (g_1h)(Y,\dots, Y^{(r)})\notin Z(K,\ell),$$ which
yields the claim by Lemma~\ref{Zw}. We now set, for $g$, $h$ as above, 
$$vs\ :=\ vg(\ell, \ell',\dots, \ell^{(r)})-vh(\ell,\dots, \ell^{(r-1)}),$$
or more suggestively, 
$$vs\ =\ v\big(G(\ell)/H(\ell)\big)\in \Gamma,\ \text{ with }G=g(Y,\dots,Y^{(r)}),\ H=h(Y,\dots, Y^{(r-1)}).$$
We thus have extended $v\colon K^\times\to\Gamma$ to a map  
$$v\ \colon\  K(y_0,\dots, y_r)^\times \to \Gamma.$$ 
Let $s\in K(y_0,\dots, y_r)^\times$ and take
$g\in K[Y_0,\dots, Y_r]$, $h\in K[Y_0,\dots, Y_{r-1}]$ with 
$g(Y, Y',\dots, Y^{(r)})\notin Z(K,\ell)$  and $h\ne 0$ such that 
$s=g(y_0,\dots, y_r)/h(y_0,\dots, y_{r-1})$. 
Let $s_1, s_2\in K(y_0,\dots, y_r)^\times$. Then $v(s_1s_2)=vs_1+vs_2$ follows
easily by means of Lemma~\ref{suplim3}. Next, assume also
$s_1+s_2\ne 0$; to prove that $v\colon K(y_0,\dots, y_r)^\times \to \Gamma$ is a valuation it remains to show that then $v(s_1+s_2) \ge \min(vs_1, vs_2)$.
For $i=1,2$ we have $s_i=g_i(y_0,\dots, y_r)/h_i(y_0,\dots, y_{r-1})$ where
$$ 0\ne g_i\in K[Y_0,\dots, Y_r],\quad 0\ne h_i\in K[Y_0,\dots, Y_{r-1}],$$
and $g_i$ has lower degree in $Y_r$ than $p$. Then for $s:= s_1+s_2$ we have 
$$s\ =\ g(y_0,\dots, y_r)/h(y_0,\dots, y_{r-1}),\qquad  g\ :=\  g_1h_2 + g_2h_1,\quad h=h_1h_2,$$
and so $g\ne 0$ (because $s\ne 0$) and $g$ has also lower degree in $Y_r$ than $p$.
In particular, $g(Y,\dots, Y^{(r)})\notin Z(K,\ell)$, hence
$vs=v\big(g(\ell,\dots, \ell^{(r)})/h(\ell,\dots, \ell^{(r-1)})\big)$, and so by working in the valued field $K\<\ell\>$ we see that $vs\ge \min(vs_1, vs_2)$, as promised.
Thus we now have $K(y_0,\dots, y_r)$ as a valued field extension of $K$.
To show that $K(y_0,\dots, y_r)$ has the same residue field as $K$, consider an
element $s=g(y_0,\dots, y_r)\notin K$ with nonzero $g\in K[Y_0,\dots, Y_r]$ of lower degree in $Y_r$ than $p$;
it suffices to show that $s\sim b$ for some $b$. Set 
$G:= g(Y, \dots, Y^{(r)})$ and take $a$ and $\fv$ with
$a-\ell \prec \fv$ and $\ndeg_{\prec \fv}G_{+a}=0$. Then 
$G(\ell)\sim G(a)$ by Lemma~\ref{Z1}, so for $b:= G(a)$ we have 
$v(s-b)=v\big(g(y_0,\dots, y_r)-b\big)=v(G(\ell)-b) > vb$, that is, $s\sim b$. 
This finishes the proof that the valued field $F:=K(y_0,\dots, y_r)$ is an immediate extension of $K$. 

\medskip\noindent
Next we equip $F$ with the derivation extending the derivation of $K$
such that $y_i'=y_{i+1}$ for $0\le i < r$. Setting $f:= y_0$ we have $f^{(i)}=y_i$
for $i=0,\dots,r$, $F=K\<f\>=K(y_0,\dots, y_r)$, and $P(f)=0$.  Note that $v(G(f))=v(G(\ell))$ for every nonzero $G\in K[Y,\dots, Y^{(r)}]$ of lower degree in $Y^{(r)}$ than $P$, in particular,
$v(f-a)=v(\ell -a)$ for all $a$. We now show that the derivation of $F$ is continuous and that $F$ is a strict extension of $K$.

\medskip\noindent
Let $\phi\in K^\times$ and $v\phi\in \Gamma(\der)$. To get 
$\der\smallo_{F}\subseteq \phi\smallo_{F}$, we set
$$S\ :=\ \big\{H(f):\ H\in K\big[Y,\dots, Y^{(r-1)}\big],\ H(f)\preceq 1\big\}.$$
(If $r=0$, then we have $K\big[Y,\dots, Y^{(r-1)}\big]=K$, so $S=\mathcal{O}$.)
By Lemma~\ref{strictimm} and by \cite[Lemma 6.2.3]{ADH} applied to~$K\big(f,\dots, f^{(r-1)}\big)$ in the role of $E$ and with $F=L$ it is enough to show
that $\der S\subseteq \phi\mathcal{O}_{F}$ and~${\der(S\cap \smallo_F)\subseteq \phi\smallo_F}$. We prove the first of these inclusions. The second follows in 
the same way. 

\medskip\noindent
Let $H\in K\big[Y,\dots,Y^{(r-1)}\big]\setminus K$ with $H(f)\preceq 1$;
we have to show $H(f)'\preceq \phi$. We can assume $H(f)'\ne 0$. 
Take  $H_1(Y), H_2(Y)\in K\big[Y,\dots,Y^{(r-1)}\big]$ such that
$$H'\ =\ H(Y)'\  =\  H_1(Y) + H_2(Y)Y^{(r)}  \quad \text{in $K\{Y\}$.}$$
Then $$H'(f)\ =\ H(f)'\ =\  
H_1(f) + H_2(f)f^{(r)},$$ and for all $a$,
$$H'(a)\ =\ H(a)'\ =\   H_1(a) + H_2(a)a^{(r)}.$$
We now distinguish two cases:  

\case[1]{$P$ has degree $>1$ in $Y^{(r)}$, or $H_2=0$.}
Then $H'$ has lower degree in $Y^{(r)}$ than $P$, so we can take $a, \fv$ with $a-\ell\prec \fv$, $\ndeg_{\prec \fv}H_{+a}=0$,
and $\ndeg_{\prec \fv}H'_{+a}=0$, so $H(a)\sim H(f)\preceq 1$ and 
$H'(a)\sim H'(f)$. Hence $H(f)'\sim H(a)'\preceq \phi$.   

\case[2]{$P$ has degree $1$ in $Y^{(r)}$ and $H_2\ne 0$.} Then
$$H'\ =\ \frac{G_1 P + G_2}{G}, \quad G_1, G_2, G\in K[Y,\dots, Y^{(r-1)}],\ G_1, G\ne 0,$$
so $0\ne H(f)'=G_2(f)/G(f)$, so $G_2\ne 0$. 
By Lemma~\ref{suplim3} there is a $\fv$ such that for some $a$ we have $a-\ell \prec \fv$ and
$$\ndeg_{\prec\fv}H_{+a}\ =\ \ndeg_{\prec\fv}(G_1)_{+a}\ =\ \ndeg_{\prec\fv}(G_2)_{+a}\ =\ \ndeg_{\prec\fv}G_{+a}=0.$$ 
Fix such $\fv$, and let $A\subseteq K$ be the set of all $a$ satisfying the above. Then for $a\in A$ we have
$G(f)\sim G(a)$ and $H(f)\sim H(a)$, so $H(a)'\preceq \phi$. Also 
$G_1(f)\sim G_1(a)$ and
\begin{align*}
G(f)H(f)'\ &=\ G_2(f)\ \sim\ G_2(a),\\
G(a)H(a)'\ &=\ G_1(a)P(a)+ G_2(a).
\end{align*}
We now make crucial use of Lemma~\ref{suplim4} to arrange that
$$v\big(G_1(a)P(a)+G_2(a)\big)\ =\ \min\big(v(G_1(a)P(a)),\ v(G_2(a)\big)$$
by changing $a$ if necessary. Hence $G_2(a)\preceq G_1(a)P(a)+G_2(a)=G(a)H(a)'$, so
$$G(f)H(f)'\ \sim\ G_2(a)\ \preceq\ G(a)H(a)'\ \sim\ G(f)H(a)'\ \preceq\ G(f)\phi$$
and thus $H(f)'\preceq \phi$. This concludes the proof that $F$ is a strict extension of $K$.

\medskip\noindent
Suppose $s$ in a strict extension $M$ of $K$ satisfies
$P(s)=0$ and $v(a-s)=v(a-\ell)$ for all $a$.  By
Lemma~\ref{Z1} and the remarks following its proof we have $vQ(s)=vQ(f)$ for all $Q\notin Z(K,\ell)$, in particular,
$Q(s)\ne 0$ for all $Q$ of lower complexity than $P$. Thus we have a
differential field embedding $K\langle f \rangle \to M$ over $K$ sending $f$ to $s$,
and this is also a valued field embedding.
\end{proof}

\subsection*{Proof of Theorem~\ref{thimm}} Assume $S(\der)= \{0\}$; we show that $K$ has an immediate strict extension 
that is maximal as a 
valued field. We can assume that $K$ itself is not yet maximal, and 
it is enough to show that then $K$ has a proper immediate strict extension, since by Lemma~\ref{glgk} the property $S(\der)=\{0\}$ is preserved by immediate strict extensions. As $K$ is not maximal, we have a divergent pc-sequence
in $K$, which pseudoconverges in an elementary extension of $K$, and 
thus has a pseudolimit~$\ell$
in a strict extension of $K$. If $Z(K,\ell)=\emptyset$, then 
Lemma~\ref{zdf} provides a proper immediate strict extension
of $K$, and if $Z(K,\ell)\ne \emptyset$, then Lemma~\ref{zda} provides 
such an extension. This concludes the proof of Theorem~\ref{thimm}.  \qed

\section{Coarsening and $S(\der)$}\label{sec:coarsening}

\noindent
In this section we finish the proof of the main theorem stated in the introduction. 

\subsection*{Making $S(\der)$ vanish} {\em In this subsection we set $\Delta:=S(\der)$ and assume $\Delta\ne \{0\}$}. Then $\Gamma(\der)$ has no
largest element, and so $v(\der\mathcal{O})>\Gamma(\der)$ by Lemma~\ref{notmax}. The next lemma says much more. Let $K_{\Delta}$ be the $\Delta$-coarsening, with valuation ring 
$\dot{\mathcal{O}}$. 

\begin{lemma}\label{derdot} 
$v(\der\dot{\mathcal{O}})>\Gamma(\der)$. 
\end{lemma}
\begin{proof} Let $a\in \dot{\mathcal{O}}$. If $va\ge 0$, then
$va'>\Gamma(\der)$ by the above. If
$va<0$, then $va\in \Delta$ and $va'-2va=v((1/a)')>\Gamma(\der)$,
so $va'>\Gamma(\der)+2va=\Gamma(\der)$. 
\end{proof}

\noindent
It follows in particular that if $\der$ is small, then the derivation of $\res(K_{\Delta})$ is trivial.  Let $\pi\colon \Gamma\to \dot{\Gamma}:=\Gamma/\Delta$ be the canonical map, so
$\pi\Gamma(\der)\subseteq \dot{\Gamma}$. We also have $\dot{\Gamma}(\der):=\dot{\Gamma}_{K_{\Delta}}(\der)\subseteq \dot{\Gamma}$, with $\pi\Gamma(\der)\subseteq \dot{\Gamma}(\der)$ by Lemma~\ref{coarseder1}.

\begin{lemma}\label{skdelta} $S_{K_{\Delta}}(\der)=\{0\}\subseteq \dot{\Gamma}$.
\end{lemma}
\begin{proof} If $\pi\Gamma(\der) = \dot{\Gamma}(\der)$, then clearly $S_{K_{\Delta}}(\der)=\{0\}$. Suppose that $\pi\Gamma(\der)\ne \dot{\Gamma}(\der)$. (We don't know if this can happen.) Then Lemma~\ref{coarseder1} tells us that $\dot{\Gamma}(\der)$ has a largest element, and so 
$S_{K_{\Delta}}(\der)=\{0\}$ by Lemma~\ref{supremum}.
\end{proof}

\subsection*{Lifting strictness} Let $K$ have small derivation and let $L$ be an immediate extension of 
$K$ with small derivation. Let $\Delta$ be a convex subgroup of $\Gamma$, giving rise to the extension $L_{\Delta}$ of $K_{\Delta}$, both with value group $\dot{\Gamma}=\Gamma/\Delta$. Note that if
$\phi\in K^\times$ and $v\phi\in \Gamma_K(\der)$, then $\phi^{-1}\der$ is small with respect to $v$, and thus small with respect to $\dot{v}$ by Lemma~\ref{lem:4.4.4}, so $\dot{v}\phi\in \dot{\Gamma}_{K_{\Delta}}(\der)$. 
We show that under various assumptions strictness of $L_{\Delta}\supseteq K_{\Delta}$ yields strictness of $L\supseteq K$:  

\begin{lemma}\label{resstrict} Suppose $L_{\Delta}$ strictly extends $K_{\Delta}$ and $\res(L_{\Delta})=\res(K_{\Delta})$. Then
$L$ strictly extends $K$.
\end{lemma}
\begin{proof} Let $\phi\in K^\times$, $v\phi\in \Gamma(\der)$ and $0\ne f\in \smallo_L$. Then $f=g(1+\epsilon)$ with
$g\in K^\times$ and $\dot{v}(\epsilon)>0$, so $vf=vg$ and
$f'=g'(1+\epsilon) + g\epsilon'$. Now
$v(g') > v(\phi)$. Since $L_{\Delta}$ strictly extends $K_{\Delta}$
we have $\dot{v}(\epsilon')>\dot{v}(\phi)$, so
$v(\epsilon') > v(\phi)$. Hence $v(f')>v(\phi)$. 
\end{proof}

\begin{lemma}\label{strictdelta} Suppose $L_{\Delta}$ strictly extends $K_{\Delta}$
and $\Delta=S(\der)\ne\{0\}$. Then $L$ strictly extends $K$.
\end{lemma}
\begin{proof} Let $0\ne f\in \smallo_L$. Then $f=gu$ with
$g\in K$ and $v(u)=0$, so $g\in \smallo$ and $f'=g'u+gu'$.
We have $v(g'u)=v(g')>\Gamma(\der)$. By Lemma~\ref{derdot}
we have 
$\dot{v}(\der\dot{\mathcal{O}})>\dot{\gamma}$ for every $\gamma\in \Gamma(\der)$. Since $L_{\Delta}$ strictly
extends $K_{\Delta}$, this gives 
$\dot{v}(\der\dot{\mathcal{O}}_L)>\dot{\gamma}$ for every $\gamma\in \Gamma(\der)$,
hence $v(\der\dot{\mathcal{O}}_L)>\Gamma(\der)$,
and so $v(u')>\Gamma(\der)$. This gives $v(f')>\Gamma(\der)$. 
\end{proof}

\subsection*{Building strict extensions by extending the residue field} 
Suppose the derivation of $K$ is small.
Let $f\in \smallo$, and let~$a$ be an element in a field extension of~$K$, transcendental over $K$. We extend the derivation of $K$ to the   derivation on $K(a)$ such that $a'=f$. We equip~$K(a)$ with the gaussian extension of the valuation of~$K$~\cite[Lem\-ma~3.1.31]{ADH}: 
the unique valuation on $K(a)$ extending the valuation of $K$ such that $a\preceq 1$ and  $\res a$ is transcendental over $\res(K)$.
So for $b=P(a)/Q(a)\in K(a)$ where
$0\neq P,Q\in K[Y]$, we have $vb=vP-vQ$; in particular, $\Gamma_{K(a)}=\Gamma$  and    $\res\big(K(a)\big)=\res(K)(\res a)$.

\begin{lemma}
The derivation of $K\<a\>$ is small. If
$\der\mathcal O\subseteq\smallo$, then   $\der\mathcal O_{K(a)}\subseteq\smallo_{K(a)}$.
\end{lemma}
\begin{proof}
Given
$P=P_dY^d + \cdots + P_0\in K[Y]$ (where $P_0,\dots,P_d\in K$), we have $P(a)'=P_d'a^d + \cdots + P_0' + f\cdot(\partial P/\partial Y)(a)$, hence $P(a)\prec 1\Rightarrow P(a)'\prec 1$, and $P(a)\preceq 1\Rightarrow P(a)'\preceq1$. Let $b\in \smallo_{K\<a\>}$. Then $b=P(a)/Q(a)$ where $P,Q\in K[Y]$ and $P(a)\prec 1\asymp Q(a)$, so $P(a)'\prec 1$ and $Q(a)'\preceq 1$, 
hence $$b'\ =\ \frac{P(a)'Q(a)-P(a)Q(a)'}{Q(a)^2}\ \prec\ 1.$$ 
Thus  $\der\smallo_{K\<a\>}\subseteq\smallo_{K\<a\>}$.
Similarly  one shows that if $\der\mathcal O\subseteq\smallo$, then $\der\mathcal O_{K(a)}\subseteq\smallo_{K(a)}$.
\end{proof}

\begin{lemma}\label{strictres} Suppose $vf>\Gamma(\der)$. Then $L:=K(a)$ is a strict extension of $K$.
\end{lemma} 
\begin{proof} Let $\phi\in K^\times$ and $\der\smallo\subseteq \phi\smallo$. 
Then the derivation of $K^\phi$ is small and $L^\phi=K^\phi(a)$ where $\phi^{-1}\der(a)=\phi^{-1}f\prec 1$. Hence by the preceding lemma applied to $K^\phi$, $\phi^{-1}f$ instead of $K$, $f$, we have
$\phi^{-1}\der\smallo_{L}\subseteq \smallo_L$ and hence
$\der \smallo_L\subseteq \phi\smallo_L$.
In the same way we show that if $\der\mathcal{O}\subseteq \phi\smallo$, then $\der\mathcal{O}_L\subseteq \phi\smallo_L$.  
\end{proof} 
 
\noindent
This leads to the following variant of 
\cite[Corollary~6.3.3]{ADH}:

\begin{cor}\label{corstrictres} Suppose $\der\mathcal{O}\subseteq \smallo$ and let $E$ be a field extension of $\res(K)$. Then there is a strict extension $L$ of $K$ such that $\Gamma_L=\Gamma$, the derivation of $\res(L)$ is trivial, and $\res(L)$ is, as a field, isomorphic to $E$ over 
$\res(K)$.
\end{cor}
\begin{proof} We can reduce to the case $E=\res(K)(y)$.   
If $y$ is transcendental over~$\res(K)$, then the corollary holds with 
$L=K(a)$ as defined above with $f=0$, by Lemma~\ref{strictres}. 
Next, suppose $y$ is algebraic over $\res(K)$, with minimum
polynomial $\overline{F}(Y)\in \res(K)[Y]$ over~$\res(K)$. Take
monic $F\in \mathcal{O}[Y]$
with image~$\overline{F}$ in $\res(K)[Y]$.
Then $F$ is irreducible in $K[Y]$. Take a field extension 
$L=K(a)$ of~$K$ where~$a$ is algebraic over $K$ with
minimum polynomial~$F$ over~$K$. 
Then there is a unique valuation $v_L\colon L^\times \to \Gamma$ that extends the valuation of~$K$; see \cite[Lemma~3.1.35]{ADH}. Then $L$  with this valuation and the  unique derivation extending the derivation of $K$  has the desired property, by Lemma~\ref{strictalg} and the remark following its proof. 
\end{proof} 

\noindent
For future reference we also state~\cite[Corollary~6.3.3]{ADH} itself:

\begin{lemma}\label{lem:6.3.3}
Let $E$ be a differential field extension of $\res(K)$. Then there is an extension $L$ of $K$ with small derivation having the same value group as $K$ and differential residue field isomorphic to $E$ over $\res(K)$.
\end{lemma}

\subsection*{Further generalities about coarsening} In this subsection we suspend our convention that $K$ denotes a 
valued {\em differential\/} field, and just assume it is a valued field, not necessarily of characteristic $0$. Notations
not involving $\der$ keep their usual meaning; in particular, the valuation of $K$ is $v\colon K^\times \to \Gamma=v(K^\times)$. Let
$\Delta$ be a convex subgroup of $\Gamma$. Then the
coarsening $K_{\Delta}$ of $K$ by
$\Delta$ is the valued field  with the same underlying field
as $K$, but with valuation
$\dot{v}=v_{\Delta}\colon K^\times \to \dot{\Gamma}=\Gamma/\Delta$.
The residue field $\res(K_\Delta)$ of $K_\Delta$ is turned into a valued field with value group $\Delta$ and residue field $\res(K)$ as described in the subsection on coarsening of Section~\ref{sec:prelim}.
The following well-known fact is \cite[Corollary~3.4.6]{ADH}, and is used several times below:

\begin{lemma}\label{lem:3.4.6}
The valued field $K$ is spherically complete iff the valued fields $K_\Delta$ and $\res(K_\Delta)$ are spherically complete.
\end{lemma}

\noindent
Let $F$ be a valued field extension of $K_{\Delta}$ with value group 
$v_F(F^\times)=\Gamma/\Delta$. Let also $\res(F)$ be given a valuation $w\colon \res(F)^\times \to \Delta$ that extends the
valuation $v\colon \res(K_{\Delta})^\times \to \Delta$. Then we can
extend $v\colon K^\times \to \Gamma$ to a map 
$v\colon F^\times \to \Gamma$  as follows. For $f\in F^\times$, take $g\in K^\times$ and $u\in F^\times$ such that
$f=gu$ and $v_F(u)=0$; then $\res u\in
\res(F)^\times$, so $w(\res u)\in \Delta$; it is 
easy to check that $v(g)+w(\res u)\in \Gamma$ depends only on~$f$ and not
on the choice of $g$, $u$; now put $v(f):= v(g)+w(\res u)$. 

\begin{lemma}\label{immcoarse} 
$v\colon F^\times \to \Gamma$ 
is a valuation on $F$ with $\Delta$-coarsening $v_{\Delta}=v_F$.
\end{lemma}
\begin{proof} Clearly 
$v\colon F^\times \to \Gamma$ is a group morphism  
with $v_F(f)=v(f)+ \Delta\in \Gamma/\Delta$ for $f\in F^\times$. 
Also, if $f\in F^\times$ and $v_F(f)>0$, then $vf>0$ and $v(1+f)=0$. 
Next, for $f_1,f_2\in F^\times$ with $f_1+f_2\ne 0$ one shows that 
$v(f_1+f_2) \ge \min\big\{vf_1,vf_2\big\}$ by distinguishing the cases
$v_F(f_1)=v_F(f_2)$ and $v_F(f_1) < v_F(f_2)$. 
\end{proof}
   
\noindent
Let $L$ be the valued field extension of $K$ that has the same underlying field as $F$ and has valuation $v$ as above. Then the lemma above says that
$L_{\Delta}=F$, and the valuation $w$ on $\res(F)$ equals
the valuation $v\colon \res(L_{\Delta})^\times \to \Delta$ induced by $v\colon L^\times \to \Gamma$ and $\Delta$. If $\res(L_{\Delta})$ is an immediate extension of $\res(K_{\Delta})$, then $L$ is an immediate extension of $K$.  See the following diagram, where
arrows like $\dashrightarrow$ indicate partial maps; for example, the residue map of~$K_\Delta$ is defined
only on $\dot{\mathcal O}$.
$$\xymatrix{ K_\Delta\ \ar@{^{(}->}[rr]\ar@{-->}[d] & & F\ar@{-->}[d] &\hskip-4.5em=\ L_\Delta \\ 
\res(K_\Delta)\ \ar@{-->}[dr]_v \ar@{^{(}->}[rr] &  &\ar@{-->}[dl]^w \, \res(F) \\
& \Delta & } $$
In the situation above, assume $K$ is of characteristic zero and is equipped with a small derivation (with respect to $v$), and $F$ is equipped with a small
derivation (with respect to $v_F$) that makes it a valued differential field extension of $K_{\Delta}$. Assume also that
the induced derivation on $\res(F)$ is small with respect to $w$.
Then the derivation of $F$ is small as a derivation of $L$ (with respect to the valuation $v$ of $L$).

\subsection*{Putting it all together} First one more special case of the main theorem:

\begin{prop}\label{adhimms} Suppose $\der$ is small and the derivation of $\res(K)$ is nontrivial. Then $K$ has the Krull property.
\end{prop} 

\noindent
In view of Lemma~\ref{strict}, this is just \cite[Corollary~6.9.5]{ADH}. We have not yet completely settled the case $S(\der)=\{0\}$ of the main theorem, but we can now take care of this:

\begin{prop}\label{sdg} Suppose $S(\der)=\{0\}$ and $\Gamma^{>}$ has a least element. Then $K$ has the Krull property.
\end{prop}
\begin{proof} Let $1$ denote the least element of $\Gamma^{>}$. We first note that $\Gamma(\der)$ has a largest element: otherwise, $\Gamma(\der)$ would be closed under adding $1$, and so $1\in S(\der)$, a contradiction. Thus by compositional conjugation we can arrange that $\Gamma(\der)=\Gamma^{\le}$, so the derivation
of $K$ is small. We have the convex
subgroup $\Delta:=\Z 1$ of $\Gamma$, so
the valuation of the
differential residue field $\res(K_{\Delta})$ of the coarsening
$K_{\Delta}$ is discrete. The completion $\res(K_{\Delta})^{\operatorname{c}}$ of the valued field $\res(K_{\Delta})$ is
a spherically complete immediate extension of $\res(K_{\Delta})$.
Since the derivation of $K$ is small, so is that of $K_\Delta$ and hence that
of $\res(K_\Delta)$. (See the remarks after Lemma~\ref{lem:4.4.4}.)
The derivation of $\res(K_{\Delta})$ is nontrivial: with $\phi\in K$ satisfying $v\phi=1$ we have $\der\smallo\not\subseteq\phi\smallo$, since  $\Gamma(\der)=\Gamma^{\le}$, so we can take  $g\in\smallo$ with $v(g')\leq v\phi=1$, and then $v_\Delta(g)\geq 0=v_\Delta(g')$. This derivation extends uniquely to a continuous derivation on $\res(K_{\Delta})^{\operatorname{c}}$, and $\res(K_{\Delta})^{\operatorname{c}}$ equipped
with this derivation is a strict extension
of the valued differential field $\res(K_{\Delta})$.  
$$\xymatrix{ K_\Delta\ \ar@{^{(}->}[r]\ar@{-->}[d]  &  F\ar@{-->}[d] &\hskip-6.5em=\ L_\Delta  \\ 
\res(K_\Delta)\   \ar@{^{(}->}[r]   &   \res(F) &\hskip-2.5em \cong \ \res(K_\Delta)^{\operatorname{c}}} $$
By applying Lemma~\ref{lem:6.3.3} to the differential field extension $\res(K_{\Delta})^{\operatorname{c}}\supseteq \res(K_{\Delta})$
we obtain an extension $F$ of $K_{\Delta}$ with small derivation,
the same value group~$v_F(F^\times)=\Gamma/\Delta$ as 
$K_{\Delta}$, and with
differential residue field $\res(F)$ isomorphic to
$\res(K_{\Delta})^{\operatorname{c}}$ over~$\res(K_{\Delta})$. Extending $F$
further using Proposition~\ref{adhimms}, if necessary, we arrange also that $F$ is spherically complete.

Next we equip
$\res(F)$ with a valuation $w\colon \res(F)^\times \to \Delta$ that makes 
$\res(F)$ isomorphic as a valued differential field to $\res(K_{\Delta})^c$ over $\res(K_{\Delta})$. This places us in the situation of the previous subsection, and so we obtain an extension $L$ of $K$ with the same value group $\Gamma$ such that
$L_{\Delta}=F$ (so $L$ and $F$ have the same underlying differential field), the valuation induced by $L$ and $\Delta$
on $\res(L_{\Delta})=\res(F)$ equals $w$, and the derivation of
$L$ is small. It follows easily that $L$ is an immediate extension of~$K$. Since $F=L_{\Delta}$ and $\res(L_{\Delta})$ are spherically complete, $L$ is spherically complete by Lemma~\ref{lem:3.4.6}. Since the derivation of $L$ is small
and $\Gamma_K(\der)$ has largest element $0$, the extension $L$ of $K$ is strict, by  Lemma~\ref{strictimm}.  
\end{proof}

\noindent
We can now finish the proof of our main theorem. We are given
$K$ and have to show that $K$ has a spherically complete
immediate strict extension. We already did this in several cases, and 
by Theorems~\ref{thimm} and Proposition~\ref{sdg} it only remains to consider the case $\Delta:=S(\der)\ne \{0\}$. We assume this below and also arrange by
compositional conjugation that the derivation is small. By
Lemma~\ref{derdot} we have $\der\dot{\mathcal{O}}\subseteq \dot\smallo$, and so the derivation of
$\res(K_{\Delta})$ is trivial. Take a spherically complete
immediate valued field extension $E$ of the valued field
$\res(K_{\Delta})$.
By Corollary~\ref{corstrictres} applied to $K_{\Delta}$ we obtain 
a strict extension $F$ of $K_{\Delta}$ with value group
$v_F(F^\times)=\Gamma/\Delta$, the derivation of $\res(F)$
is trivial, and $\res(F)$, as a field, is isomorphic to $E$ over
$\res(K_{\Delta})$. We equip $\res(F)$ with a valuation 
$w\colon \res(F)^\times\to \Delta$ that makes $\res(F)$ isomorphic as a valued field to $E$ over $\res(K_{\Delta})$.
We are now in the situation of the previous subsection, and so we obtain an extension $L$ of $K$ with the same value group $\Gamma$ as~$K$ such that
$L_{\Delta}=F$ (so $L$ and $F$ have the same underlying differential field), the valuation induced by $L$ and $\Delta$
on $\res(L_{\Delta})=\res(F)$ equals $w$, and the derivation of
$L$ is small. Now $\res(L_{\Delta})$ is an immediate extension of $\res(K_{\Delta})$, hence $L$ is an immediate extension of $K$, and so $L$  strictly extends $K$ by Lemma~\ref{strictdelta}. 
$$\xymatrix{ K_\Delta\ \ar@{^{(}->}[r]^{\text{strict}}\ar@{-->}[d]  & L_\Delta 
\ \ar@{^{(}->}[r]^{\text{strict}}\ar@{-->}[d] & M_\Delta\ar@{-->}[d]   \\ 
\res(K_\Delta)\  \ar@{^{(}->}[r]   &  \, \res(L_\Delta) \ar@{=}[r] & \res(M_\Delta)} $$
Lemma~\ref{skdelta} yields $S_{K_{\Delta}}(\der)=\{0\}$, and so $S_{L_{\Delta}}(\der)=\{0\}$ by Lemma~\ref{glgk}. Then
Theorem~\ref{thimm} and Proposition~\ref{sdg} yield a spherically complete immediate strict extension $G$ of $L_{\Delta}$. This places us again
in the situation of the previous subsection, with $L$
and $G$ in the role of $K$ and $F$. Hence we obtain 
an extension $M$ of $L$ with the same value group $\Gamma$ as $L$ such that
$M_{\Delta}=G$ (so $M$ and $G$ have the same underlying differential field), the valuation induced by $M$ and $\Delta$
on $\res(M_{\Delta})=\res(G)=\res(F)$ equals $w$, and the derivation of $L$ is small. Therefore $M$ is an immediate extension of~$L$ and thus of $K$. Since $M_{\Delta}$ and $\res(M_{\Delta})$ are spherically complete, $M$ is spherically complete by Lemma~\ref{lem:3.4.6}. The
extension $M$ of $L$ is strict by Lemma~\ref{resstrict}. Thus
$M$ is a spherically complete immediate strict extension of $K$ as
required. This concludes the proof of the main theorem. \qed

\section{Uniqueness}\label{Uniqueness}

\noindent
Let us say that {\it $K$ has the uniqueness property\/} if it has up to isomorphism over~$K$ a unique spherically complete immediate strict
extension. If $\Gamma=\{0\}$ and more generally, if $K$ is spherically complete, then $K$ clearly has  the uniqueness property. If $\der=0$, then the derivation of any immediate strict extension of $K$ is also trivial, so $K$ has the uniqueness property. The next result describes a more interesting situation where $K$ has the uniqueness property.

\begin{prop} Suppose $\Gamma=\Z$. Then $K$ has the uniqueness property.
\end{prop}
\begin{proof} Let $\hat{K}$ be the completion of the discretely valued field $K$. Then the unique extension of $\der$ to a continuous function $\hat{K} \to \hat{K}$ is a derivation on $\hat{K}$ that makes~$\hat{K}$ an immediate strict extension of 
$K$.
If $L$ is any spherically complete immediate extension of $K$, then we have a unique valued
field embedding $\hat{K} \to L$ over $K$, and this embedding is clearly an isomorphism of valued differential fields.
\end{proof}

\begin{prop}\label{coarseder3} Suppose $\Delta$ is a convex subgroup of $\Gamma$ and $\res(K_{\Delta})$ is spherically complete. If $K_{\Delta}$ has the uniqueness property, then so does $K$.  
\end{prop}
\begin{proof} Let $L$ and $M$ be spherically complete immediate strict extensions of $K$. Then $\res(L_{\Delta})$ and $\res(M_{\Delta})$ are immediate valued field extensions of $\res(K_{\Delta})$
and thus equal to $\res(K_{\Delta})$. Hence $L_{\Delta}$ and $M_{\Delta}$ are spherically complete immediate extensions of
$K_{\Delta}$, and $L_{\Delta}$ and $M_{\Delta}$ are strict extensions of $K_{\Delta}$ by Lemma~\ref{coarseder2}. 

Next, let $i\colon L_{\Delta} \to M_{\Delta}$ be an isomorphism over $K_{\Delta}$; it is enough to show that then $i\colon L\to M$ is an
isomorphism over $K$. For $a\in L^\times$ we have
$a=b(1+\epsilon)$ with $b\in K^\times$ and $\epsilon\in \dot{\smallo}_L$, so $i(a)=b(1+i(\epsilon))$ and $i(\epsilon)\in \dot{\smallo}_M$, hence $va=vb=vi(a)$.   
\end{proof} 

\noindent
One could try to use this last result inductively, but at this stage we do not even know if uniqueness holds when $\Gamma=\Z^2$, lexicographically ordered. 

\subsection*{The role of linear surjectivity} In the next section we give an example of an $H$-field $K$ that doesn't have the uniqueness property. This has to do with the fact that certain
linear differential equations over this $K$ have no solution in $K$. Here we focus on the opposite situation: as in \cite[Section~5.1]{ADH} a differential field
$E$ of characteristic zero is said to be {\em linearly surjective\/} if for all
$a_1,\dots, a_n, b\in E$ the linear differential equation
$$y^{(n)} + a_1y^{(n-1)} + \cdots + a_n y\ =\ b $$
has a solution in $E$. For valued differential fields this property is related to differential-henselianity: we say that
$K$ is {\em differential-henselian\/} (for short: $\d$-hen\-selian)
if $K$ has small derivation and every differential polynomial $P\in \mathcal{O}\{Y\}=\mathcal{O}[Y, Y', Y'',\dots]$ whose reduction 
$\bar{P}\in \res(K)\{Y\}$ has degree~$1$ has a zero in $\mathcal{O}$; cf.~\cite[Chapter~7]{ADH}. If $K$ is $\d$-henselian, then its differential residue field $\res(K)$ is clearly linearly surjective. Here is a differential analogue of Hensel's Lemma:

 {\em If $K$ has small derivation, $\res(K)$ is linearly surjective, and $K$ is spherically complete, then $K$ is
$\d$-henselian}. This is \cite[Corollary 7.0.2]{ADH}; the
case where $K$ is monotone goes back to Scanlon~\cite{Sc}.

\begin{conjecture}  If $K$ has small derivation and $\res(K)$ is linearly surjective, then $K$ has the uniqueness property.
\end{conjecture}

\noindent
For monotone $K$ this conjecture has been established: \cite[Theorem 7.4]{ADH}. It has also been proved for $K$ whose value group has finite archimedean rank
and some related cases in \cite{DP}. Recently, Nigel Pynn-Coates has proved the conjecture in the case of most interest to us, namely for asymptotic $K$.
This is part of work in progress.

\section{Nonuniqueness}\label{Nonuniqueness}

\noindent
We begin with a general remark. Let $A\in K[\der]$ and suppose
the equation $A(y)=1$ has no solution in any immediate strict
extension of $K$. Assume in addition that $a\in K$ is such that
the equation $A(y)=a$ has a solution $y_0$ in an immediate strict extension $K_0$ of $K$ and the equation $A(y)=a+1$ has a solution $y_1$ in an immediate strict extension $K_1$ of $K$. Extending $K_0$ and $K_1$ we arrange that $K_0$ and $K_1$ are spherically complete, and we then observe that $K_0$ and $K_1$ cannot be isomorphic over $K$. Thus $K$ does not have the uniqueness property. 

Below we indicate a real closed $H$-field $K$ where the above assumptions hold for a certain $A\in K[\der]$ of order $1$, and so this $K$ does not have the uniqueness property. 

The first two subsections contains generalities about solving linear differential equations of order $1$ in immediate extensions of $\d$-valued fields. In the last subsection we assume familiarity with \cite[Sections~5.1, 11.5, 11.6, 13.9, Appendix~A]{ADH}. 

We recall from \cite[Section~9.1]{ADH} that an asymptotic field $K$ is said to be {\it $\d$-valued}\/ (short for: ``differential-valued'') if $\mathcal O=C+\smallo$. (So each $H$-field is $\d$-valued.)
We also recall that if $K$ is an asymptotic field, then for $f\in K^\times$ with $ f\nasymp 1$, the valuation $v(f^\dagger)$ of the logarithmic derivative of $f$ only depends on~$vf$, so we have a function $\psi\colon\Gamma^{\neq}:=\Gamma\setminus\{0\}\to \Gamma$ with $\psi(vf)=v(f^\dagger)$ for such $f$.
If we want to stress the dependence on $K$ we write $\psi_K$ instead of $\psi$, and for $\gamma\in\Gamma^{\neq}$ we also set $\gamma':=\gamma+\psi(\gamma)$. The pair $(\Gamma,\psi)$ is an {\it asymptotic couple,}\/ that is (see \cite[Section~6.5]{ADH}):
$\psi(\alpha+\beta)\geq\min\big\{\psi(\alpha),\psi(\beta)\big\}$ for all $\alpha,\beta\in\Gamma^{\neq}$ with $\alpha+\beta\neq 0$; $\psi(k\gamma)=\psi(\gamma)$ for $\gamma\in\Gamma^{\neq}$ and $0\neq k\in\Z$; and 
$$\Psi:=\big\{\psi(\gamma):\gamma\in\Gamma^{\neq}\big\}\ <\ (\Gamma^>)':=\big\{\gamma':\gamma\in\Gamma^>\big\}.$$

\subsection*{Slowly varying functions}
In this subsection $K$ is an asymptotic field,  
$\Gamma\neq\{0\}$, and $A\in K[\der]$ is of order $1$.
Proposition~\ref{cor:9.7.1 variant} below is
a variant of \cite[Proposition~9.7.1]{ADH}. Recall from \cite[Section~9.7]{ADH} that for an ordered abelian group $G$ and $U\subseteq G$ a function $\eta\colon U\to G$ is said to be {\it slowly varying}\/ if $\eta(\alpha)-\eta(\beta)=o(\alpha-\beta)$ for all $\alpha\neq\beta$ in $U$; note that then $\gamma\mapsto \gamma+\eta(\gamma)\colon U\to G$ is strictly increasing. Note also that $\psi\colon \Gamma^{\ne}\to \Gamma$ is slowly varying~\cite[Lemma~6.5.4(ii)]{ADH}.

\begin{lemma}\label{lem:v(ydagger-s), variant 1} 
Let  $a\in K^\times$ and $s=a^\dagger$. Then there is a slowly varying function
$\eta\colon\Gamma\setminus\{va\}\to\Gamma$ such that 
$v(y^\dagger-s)=\eta(vy)$ for all $y\in K^\times$ with $vy\neq va$.
\end{lemma}
\begin{proof}
We can take $\eta(\gamma):=\psi(\gamma-va)$ for $\gamma\in\Gamma\setminus\{va\}$.
\end{proof}

\begin{lemma}\label{lem:v(ydagger-s), variant 2} Assume $K$ is $\d$-valued.
Let $s\in K$ be such that $v(y^\dagger-s)<(\Gamma^>)'$ for all $y\in K^\times$. 
Then there 
is a slowly varying
function $\eta\colon\Gamma\to\Gamma$ such that 
$$\eta(vy)=v(y^\dagger-s)\quad\text{for all  $y\in K^\times$.}$$
\end{lemma}
\begin{proof}
Let $y$ range over $K^\times$.
Take a nonzero $\phi$ in an elementary extension $L$ of~$K$ such that
$\phi^\dagger-s\preceq y^\dagger -s$ for all $y$; thus $\delta:=v(\phi^\dagger-s)<\big(\Gamma_L^{>}\big)'$.
From $v(y^\dagger-s) \le v(\phi^\dagger -s)$ we get
$y^\dagger-\phi^\dagger\not\sim s-\phi^\dagger$, and thus
$$v(y^\dagger-s)\ =\ v\big((y^\dagger-\phi^\dagger)-(s-\phi^\dagger)\big)\ =\ \min\big\{v\big((y/\phi)^\dagger\big), \delta\big\}\ =\ \min\big\{\psi_L(vy-v\phi), \delta\big\},$$
where in case $y\asymp \phi$ we use that $L$ is $\d$-valued to
get the last equality. Thus $v(y^\dagger-s)=\eta(vy)$, where $\eta\colon \Gamma\to \Gamma$ is defined by $\eta(\gamma):= \min\big\{\psi_L(\gamma-v\phi), \delta\big\}$.  
Next we show that $\eta$ is slowly varying. The function $\gamma\mapsto \psi_L(\gamma-v\phi)\colon \Gamma_L\setminus\{v\phi\}\to\Gamma_L$ is slowly varying, hence so is
the restriction of $\eta$ to $\Gamma\setminus\{v\phi\}$. Moreover, if $v\phi\in \Gamma$ and
$\gamma\in \Gamma\setminus \{v\phi\}$, then $\eta(v\phi)=\delta$, so
\begin{align*}
\eta(\gamma)-\eta(v\phi)\ &=\ \min\big\{\psi_L(\gamma-v\phi),\delta\big\}-\delta \\
&=\  \min\big\{\psi_L(\gamma-v\phi)-\delta,0\big\}\ =\ o(\gamma-v\phi)
\end{align*}
by \cite[Lemma~9.2.10(iv)]{ADH} applied to the asymptotic couple 
$(\Gamma_L,\psi_L-\delta)$, which has small derivation.   
\end{proof}

\begin{lemma}\label{prop:9.7.1 variant} Suppose $K$ is $\d$-valued and $\big\{f\in K:\ vf\in (\Gamma^>)'\big\}\subseteq (K^\times)^\dagger$.
Then there is a slowly varying function $\eta\colon\Gamma\setminus v(\ker A)\to\Gamma$ such that
$$v\big(A(y)\big)=vy+\eta(vy)\quad\text{for all $y\in K$ with $vy\notin v(\ker A)$.}$$
\end{lemma}
\begin{proof}
We have $A=a_0+a_1\der$ with $a_0,a_1\in K$, $a_1\neq 0$;
put $s:=-a_0/a_1$.
For $y\in K^\times$ we get
$A(y)=a_1y(y^\dagger-s)$, hence
$v\big(A(y)\big)=va_1+vy+v(y^\dagger-s)$, and
the claim  follows from Lemmas~\ref{lem:v(ydagger-s), variant 1} and~\ref{lem:v(ydagger-s), variant 2}.
\end{proof}

\noindent
We refer to \cite[Section~11.1]{ADH} for the definition of the subset 
$\exc^{\ev}(A)$ of $\Gamma$, for ungrounded $K$; since $A$ has order $1$, this set $\exc^{\ev}(A)$ 
has at most one element. Recall also that $K$ is said to be of {\it $H$-type}\/ or {\it $H$-asymptotic}\/ if
   $\psi$ restricts to a decreasing function~$\Gamma^>\to\Gamma$,
and to have {\it asymptotic integration}\/ if $(\Gamma^{\neq})'=\Gamma$.

\begin{prop}\label{cor:9.7.1 variant}
Let  
$K$ be $\d$-valued of $H$-type with asymptotic integration.
Then there is a slowly varying function $\eta\colon\Gamma\setminus \exc^{\ev}(A)\to\Gamma$ such that
$$v\big(A(y)\big)\ =\ vy+\eta(vy)\quad\text{for all $y\in K^\times$ with $vy\notin \exc^{\ev}(A)$.}$$
\end{prop}
\begin{proof}
By \cite[Lem\-ma~10.4.3]{ADH} we have an immediate $\d$-valued extension $L$ of $K$
such that $\big\{s\in L:\ vs\in (\Gamma_L^>)'\big\}\subseteq (L^\times)^\dagger$.
Applying Lemma~\ref{prop:9.7.1 variant} to $L$ in place of $K$
yields 
 a slowly varying function $\eta\colon\Gamma\setminus v(\ker_L A)\to\Gamma$ such that
$$v\big(A(y)\big)\ =\ vy+\eta(vy)\quad\text{for all $y\in K$ with $vy\notin v(\ker_L A)$.}$$
It only remains
to note that $v\big((\ker_L  A)\setminus\{0\}\big)\subseteq \exc_L^{\ev}(A)=\exc^{\ev}(A)$.
\end{proof}

\subsection*{Application to solving first-order linear differential equations}
In this subsection $K$ is $\d$-valued, $A\in K[\der]$ has 
order $1$, and $g\in K$ is such that $g\notin A(K)$,
so $S:=v\big(A(K)-g\big)\subseteq \Gamma$.

\begin{lemma}\label{lem:solve order 1}
Suppose $K$ is henselian of $H$-type with asymptotic integration. Also assume 
$\exc^{\ev}(A)=\emptyset$ 
and  $S$
does not have a largest element.
Let $L=K(f)$ be a field extension of $K$
with~$f$ transcendental over~$K$, equipped
with the unique derivation extending that of $K$ such that $A(f)=g$.
Then there is a valuation of $L$ that makes~$L$ an 
immediate asymptotic extension of~$K$.
\end{lemma}
\begin{proof}
Take a well-indexed sequence $(y_\rho)$ in $K$ such that 
$\big(v\big(A(y_\rho)-g\big)\big)$ is strictly increasing and cofinal
in $S$. Proposition~\ref{cor:9.7.1 variant} yields a strictly increasing function $i\colon\Gamma\to\Gamma$
with $v\big(A(y)\big)=i(vy)$ for all  $y\in K^\times$. Hence
for $\rho<\sigma$,
$$v\big(A(y_\rho)-g\big)\ =\ 
v\big((A(y_{\rho})-g)-(A(y_{\sigma})-g)\big)\ =\ 
v\big(A(y_\rho-y_\sigma)\big)\ =\ 
i\big(v(y_\rho-y_\sigma)\big),$$ 
so
$i\big(v(y_\rho-y_\sigma)\big) < i\big(v(y_\sigma-y_\tau)\big)$ and
thus $v(y_\rho-y_\sigma)<v(y_\sigma-y_\tau)$
for $\rho<\sigma<\tau$. Hence $(y_\rho)$ is a pc-sequence. 
Suppose towards a contradiction that $y_{\rho} \leadsto y\in K$. 
Then $v(y_\rho-y)$ is eventually strictly increasing, so 
$v\big(A(y_\rho)-A(y)\big)=i\big( v(y_\rho-y) \big)$ is eventually strictly increasing, and thus eventually
$v\big( A(y_{\rho}) - g \big) \leq v\big( A(y)-g \big)$,
contradicting the assumption that $S$ has no largest element. Hence $(y_\rho)$ does not have a pseudolimit in~$K$.
It remains to use \cite[Proposition~9.7.6]{ADH}. 
\end{proof}

\noindent
Here is a situation where the hypothesis about $S$ in Lemma~\ref{lem:solve order 1} is satisfied:

\begin{lemma} \label{lem:no max}
If $S\subseteq v\big(A(K)\big)$, then $S$ does not have a largest element.
\end{lemma}
\begin{proof}
Let $y\in K$ be given; we need to find $y_\new\in K$ with
$A(y_\new)-g\prec A(y)-g$.
Since $v\big(A(y)-g\big)\in v\big(A(K)\big)\cap\Gamma$, we
can pick $h\in K^\times$ such that $A(h)\sim A(y)-g$.
Set $y_\new:=y-h$. Then
$A(y_\new)-g=\big(A(y)-g\big)-A(h)\prec A(y)-g$
as required. 
\end{proof}

\subsection*{Some differential-algebraic lemmas}
In this subsection $E$ is a differential field of characteristic zero and $F$ is a differential field extension of $E$.

\begin{lemma}\label{lem:trace} Let $F$ be algebraic over $E$, and $f'+af=1$ with $a\in E$ and $f\in F$. Then $g'+ag=1$ for some $g\in E$.
\end{lemma}
\begin{proof}
We can assume that $n:=[F:E]<\infty$. The trace map $\operatorname{tr}_{F|E}\colon F\to E$ is $E$-linear and satisfies
$\operatorname{tr}_{F|E}(y')=\operatorname{tr}_{F|E}(y)'$ for all $y\in F$ and $\operatorname{tr}_{F|E}(1)=n$.
Thus $g:=\frac{1}{n}\operatorname{tr}_{F|E}(f)\in E$ satisfies $g'+ag=1$.
\end{proof}

\begin{lemma}\label{lem:diff trans}
Let $F=E\<y\>$ where $y$ is differentially transcendental over $E$, and let $a\in E(y)$. Then there is no $f\in F\setminus E$ with $f'+af=1$.
\end{lemma}
\begin{proof} This is a special case of \cite[Lemma 4.1.5]{ADH}.
\end{proof}

\begin{lemma}
Let $Y$ be an indeterminate over a field $G$ and let $R\in G(Y)$ be such that $R(Y)=R(Y+g)$ for infinitely many $g\in G$. Then $R\in G$.
\end{lemma}
\begin{proof} We have $R=P/Q$ with $P,Q\in G[Y]$. Let $Z$ be an indeterminate over $G(Y)$. Then 
$R(Y)=R(Y+g)$ for infinitely many $g\in G$ yields  $$P(Y)Q(Y+Z)\ =\ Q(Y)P(Y+Z).$$
Substituting $g-Y$ for $Z$ yields $P(Y)Q(g)=Q(Y)P(g)$ for all $g\in G$.
Choosing $g$ such that $Q(g)\neq 0$, we obtain $R(Y)=P(Y)/Q(Y)=P(g)/Q(g)\in G$.
\end{proof}

\begin{cor}\label{cor:int}
Let $F=E(y)$ with  $y'\in E\setminus\der E$, and let $a\in E\setminus (E^\times)^\dagger$. Then there is no $f\in F\setminus E$ with $f'+af=1$.
\end{cor}
\begin{proof}
By \cite[Lemma~4.6.10]{ADH}, $y$ is transcendental over $E$, and by \cite[Corollary~4.6.13]{ADH} there is no $g\in F^\times$ with $g'+ag=0$. 
For each $c\in C_E$ we have an automorphism~$\sigma_c$ of the differential field $E(y)$ which is the identity on $E$ and sends~$y$ to~$y+c$. Suppose $f'+af=1$, $f\in F$. Then $\big(f-\sigma_c(f)\big)'+a\big(f-\sigma_c(f)\big)=0$ and hence
$\sigma_c(f)=f$, for each $c\in C_E$. Hence $f\in F$ by the preceding lemma.
\end{proof}

\subsection*{Non-isomorphic spherically complete extensions}
We now use the preceding subsections to construct an $H$-field $K$ with two spherically complete immediate $H$-field extensions that are not isomorphic over $K$. Let $\mathfrak M$ be the subgroup of the ordered multiplicative group $G^{\text{LE}}$ of $\text{LE}$-monomials
generated by the rational powers of~$\ex^x$ and the iterated logarithms $\ell_n$ of~$x$:
$$\mathfrak M\ =\ \bigcup_n \ex^{\Q x} \ell_0^\Q \cdots \ell_n^\Q.$$
We consider the spherically complete ordered valued Hahn field 
$$M\ :=\ \R[[\mathfrak M]]\ \subseteq\ \R[[G^{\text{LE}}]].$$ 
Note that $\mathfrak L:= \bigcup_n   \ell_0^\Q \cdots \ell_n^\Q$ is a convex subgroup of $\mathfrak M$ with $\mathfrak L\cap \ex^{\Q x}=\{1\}$ and $\mathfrak{M}=\mathfrak{L}\ex^{\Q x}$, and so $M=\mathbb L[[\ex^{\Q x}]]$ where $\mathbb L=\R[[\mathfrak L]]$. 
(Our use of the symbols~$\mathfrak L$,~$\mathbb L$ differs slightly from that in \cite[Section~13.9]{ADH}.) 
We 
equip $M$ with the unique strongly $\R$-linear derivation satisfying  
$$(\ex^{rx})'=r\ex^{rx}, \quad (\ell_0^r)'\ =\ r\,\ell_0^{r-1}, \quad (\ell_{n+1}^r)'\ =\ r\,\ell_{n+1}^{r-1}(\ell_0\cdots \ell_n)^{-1}\qquad (r\in \Q).$$
Then $M$ is an $H$-field with constant field $\R$. The element $\upl\in \mathbb{L}$ is defined by
$$\upl\, :=\, \left(\sum_{n=1}^{\infty}\ell_n\right)'\, =\,\sum_{n=0}^{\infty}(\ell_0\cdots\ell_n)^{-1},$$  
as in \cite[Section~13.9]{ADH}. Consider the real closed $H$-subfield $E := \R\<\upl,\ell_0,\ell_1,\dots\>^{\operatorname{rc}}$  of~$\mathbb L$ and the real closed $H$-subfield
$K := E[[\ex^{\Q x}]]$ of $M$. Note that $\mathbb L$ is an immediate extension of $E$ and $M$ is an 
immediate extension of $K$. Thus $K$ has the same divisible value group  
$\Q v(\ex^x)\oplus \bigoplus_n \Q v(\ell_n)$ as $M$, and $K$ has asymptotic integration. Note also that $(\ell_n)$ is a logarithmic  sequence in $K$ in the sense of \cite[Section~11.5]{ADH}. 

We set $A :=\ \der-\upl\in E[\der]$. 
Let $K^*$ be an immediate $H$-field extension of $K$. By~\cite[Lemma~11.5.13]{ADH} we have 
$\ker_{K^*} A = \{0\}$. Moreover, $-\upl$ creates a gap over~$K^*$,
by~\cite[Lemma~11.5.14]{ADH} and so $A(y) \nasymp 1$ for all $y\in K^*$,
by~\cite[Lemma~11.5.12]{ADH}; in particular $1\notin A(K^*)$.
These remarks apply in particular to $K^*=M$.
We are going to show:

\begin{prop}\label{prop:non-iso}
For every $c\in \R$ there is an element $y$ in some immediate $H$-field extension $K_c$ of $K$ with $A(y)=\ex^x+c$.
\end{prop}

\noindent
By Lemma~\ref{as1}, any immediate $H$-field extension of $K$ strictly extends $K$.
Thus in view of the remark in the beginning of this section and
using Proposition~\ref{prop:non-iso}:

\begin{cor} There is a family $(K_c)_{c\in \R}$ of spherically complete immediate strict $H$-field extensions $K_c$ of $K$ that
are pairwise non-isomorphic over $K$. 
\end{cor}

\noindent
In particular, $K$ does not have the uniqueness property.
Towards the proof of the proposition, we still need two lemmas. 

\begin{lemma}\label{lem:alg indep}
The elements $\ell_0,\ell_1,\dots$ of $\mathbb{L}$ are algebraically independent over the subfield $\R\<\upl\>=\R(\upl,\upl',\dots)$ of $\mathbb{L}$.
\end{lemma}
\begin{proof} The element $\upl$ is differentially transcendental over $\R$
by \cite[Corollary~13.6.3]{ADH}, and hence over 
$\R(\ell_0,\ell_1,\dots)$, so $\upl, \upl', \upl'',\dots$ are algebraically independent over $\R(\ell_0, \ell_1,\dots)$.
Since $\ell_0,\ell_1,\dots$ are algebraically independent over
$\R$, $$\ell_0, \ell_1, \ell_2,\dots, 
\upl, \upl', \upl'',\dots$$
are algebraically independent over $\R$. Hence $\ell_0, \ell_1,\dots$ are algebraically independent over $\R(\upl, \upl',\dots)$.
\end{proof}

\noindent
Let $B:=\der+(1-\upl)\in E[\der]$. 
We have $\upl\notin (M^\times)^\dagger$ by \cite[Lemma~11.5.13]{ADH} and $1=(\ex^x)^\dagger\in (M^\times)^\dagger$, so $1-\upl\notin (M^\times)^\dagger$, that is,
$\ker_M B = \{0\}$.

\begin{lemma}\label{1notBE}
$1\notin B(E)$.
\end{lemma}
\begin{proof}
Put $L_0:=\R\<\upl\>$ and $L_{n+1}:=\R\<\upl,\ell_0,\dots,\ell_{n}\>$, so  $L_{n+1}=L_n(\ell_n)$ in view of $\ell_n'=\ell_{n-1}^\dagger\in L_n$ for $n\ge 1$, and $\ell_0'=1\in L_0$. Note that
$E$ is algebraic over $\R\<\upl,\ell_0,\ell_1,\dots\>=\bigcup_n L_n$. By
Lemma~\ref{lem:trace} it suffices
that $1\notin B(L_n)$ for all $n$.
The case $n=0$ follows from Lemma~\ref{lem:diff trans}.
Suppose $1\notin B(L_n)$. Now
$L_{n+1}=L_n(\ell_n)$ and~$\ell_n$ is transcendental over $L_n$, by Lemma~\ref{lem:alg indep}, so $1\notin B(L_{n+1})$ by Co\-rollary~\ref{cor:int}.
\end{proof}

\begin{proof}[Proof of Proposition~\ref{prop:non-iso}]
Let $c\in \R$ and $g:=\ex^x+c\in K$.

\claim[1]{$A(y)\ne g$ and $A(y)-g \comp \ex^x$, for all $y\in K$.}

\noindent
This is obvious for $y=0$, so assume $y\in K^\times$. Let $r$ range over $\Q$ and let the $y_r\in E$ be such that $y=\sum_r y_r \ex^{rx}$ with the reverse-well-ordered  set $\{r:y_r\neq 0\}$ having
largest element $r_0$. Then $$A(y) =\sum_r \big(y_r'+(r-\upl) y_r\big)\ex^{rx}.$$ For $r_0\ne 0$ we have 
$r_0-\upl\asymp 1$, so $r_0-\upl\notin (\mathbb L^\times)^\dagger$, and thus for $r_0>1$,
$$A(y)-g\ \sim\ 
\big(y_{r_0}'+(r_0-\upl) y_{r_0}\big)\ex^{r_0x}\   \comp\  \ex^{x}.$$
Next, assume $r_0=1$. By Lemma~\ref{1notBE}
we have $y_1'+(1-\upl)y_1-1\ne 0$,
and thus 
$$A(y)-g\sim \big(y_1'+(1-\upl)y_1-1\big)\ex^x \comp \ex^x.$$  
Finally, if $r_0<1$, then $A(y)-g\sim -g \asymp \ex^x$. 

\medskip
\noindent
Since $K$ is an $H$-field with asymptotic integration we can pick
for every $f\in K^\times$ an element $If\in K^\times$ with $If\nasymp 1$ and $(If)' \sim f$.

\claim[2]{Suppose $f\in K^\times$ and $f\comp \ex^x$. Then $If\asymp f$.}

\noindent
To prove this, note that $h^\dagger \preceq 1$ for all $h\in M^\times$, hence
$f/If \sim (If)^\dagger \preceq 1$ and so $f \preceq If$.
If $f \prec If$, then $f' \prec  (If)' \sim f$, whereas
$f\comp \ex^x$ means $f^\dagger \asymp (\ex^x)^\dagger=1$, a contradiction. Thus $f\asymp If$, as claimed.

\medskip
\noindent
Let $y\in K$ be given, and set $z:=A(y)-g$ and
$y_{\operatorname{new}}:= y-Iz\in K$. Then
$$z_{\operatorname{new}}\ :=\ A(y_{\operatorname{new}})-g\  =\ 
 z - (\textstyle I z)' + \upl I z.$$
By Claim~1 we have $z\comp \ex^x$, so  $I z\asymp z$ by Claim~2,
and thus $\upl I z \prec z$.
Since
$z-(Iz)' \prec z$, this yields $z_{\operatorname{new}}\prec z$.

\medskip
\noindent
This argument shows that the subset
$v\big( A(K)-g \big)$
of $\Gamma$ does not have a largest element. By \cite[Example at end of Section 11.1, Lemma 11.5.13]{ADH} we have 
$\exc^{\ev}_K(A)=\emptyset$. Thus  
Proposition~\ref{prop:non-iso} follows from Lemma~\ref{lem:solve order 1}.
\end{proof}

\noindent
To finish this paper we indicate how the operator $B$ differs
in its behavior on $E$ from that on its immediate extension $\mathbb{L}$. This uses the following:

\begin{lemma}\label{11.6.15v}
Let $L$ be an $H$-asymptotic field with asymptotic integration and divisible value group $\Gamma_L$, and let $s\in L$ be such that 
$$S:=\big\{v(s-a^\dagger):\ a\in L^\times\big\}\ \subseteq\ \Psi_L^{\downarrow}.$$
Then the following are equivalent for $g\in L^\times$: \begin{enumerate}
\item[(i)] $vg\notin v\big(D(L)\big)$ for $D:=\der-s\in L[\der]$;
\item[(ii)] $g^\dagger -s$ creates a gap over $L$.
\end{enumerate}
\end{lemma}
\begin{proof} If $S$ has no largest element, this is 
\cite[Lemma~11.6.15]{ADH}. Suppose $S$ has a largest element. Then
\cite[Lemma~10.4.6]{ADH} yields an $H$-asymptotic extension $L(b)$ with
$b\ne 0$, $b^\dagger = s$, $\eta:= vb\notin \Gamma_L$, and
$\Gamma_{L(b)}=\Gamma \oplus \Z \eta$, and 
$\Psi_{L(b)} = \Psi_L\cup \{\max S\}\subseteq \Psi_L^{\downarrow}$. The rest of the argument is as in the proof of \cite[Lemma~11.6.15]{ADH}.
\end{proof} 

\noindent
In contrast to Lemma~\ref{1notBE} we have:

\begin{prop}\label{BLL}
$B(\mathbb L)=\mathbb L$; in particular $1\in B(\mathbb L)$.
\end{prop}
\begin{proof} Set $s:=\upl-1$. We have $\upl\prec 1$ and for 
$a\in \mathbb{L}^\times$ we have $a^\dagger \prec 1$. Thus
$$\big\{v(s-a^\dagger):\ a\in \mathbb{L}^\times\big\}\ =\ \{0\}\ \subseteq\ \Psi_{\mathbb{L}}^{\downarrow}.$$
Let $g\in \mathbb L^\times$. Applying Lemma~\ref{11.6.15v} yields: 
$$vg\notin v\big(B(\mathbb L)\big)\ \Longleftrightarrow\ g^\dagger-s\ \text{ creates a gap over }\ \mathbb L.$$
We have $\upl_n \leadsto \upl$. If $vg\notin v\big(B(\mathbb L)\big)$, then $\upl_n\leadsto s+g^\dagger$ by 
\cite[11.5.12]{ADH} and the above equivalence, so
$v(1+g^\dagger)>\Psi_{\mathbb{L}}$ by \cite[Lemma 11.5.2]{ADH}. But $g^\dagger\prec 1$, so
$v(1+g^\dagger)=0\in \Psi_{\mathbb{L}}^{\downarrow}$. Thus
$v\big(B(\mathbb L)\big)=v(\mathbb L^\times)$.
As we saw,
$v(g^\dagger-s)\in \Psi_{\mathbb{L}}^{\downarrow}$ for all 
$g\in \mathbb{L}^\times$, so 
$\exc^{\ev}_{\mathbb{L}}(B)=\emptyset$, by 
\cite[Example at end of Section~11.1]{ADH}.
 The desired result now follows from Lemmas~\ref{lem:solve order 1} and \ref{lem:no max} and the spherical completeness of $\mathbb L$. 
\end{proof}

\noindent
As a consequence of Proposition~\ref{BLL} we have
$\ex^x\in A(M)$: taking $y\in \mathbb L$ with $B(y)=1$ gives $A(y\ex^x)=\ex^x$. In view of the remarks just before Proposition~\ref{prop:non-iso} we also obtain that $\ex^x+c\notin A(M)$ for all nonzero $c\in \R$.

\end{document}